\documentclass[11pt,reqno]{amsart}

\usepackage{graphicx}         %

\usepackage{comment}

\usepackage{amsmath, amsopn, amssymb}
\usepackage{a4wide}
\usepackage{xspace}
\usepackage{float}
\usepackage{natbib}

\usepackage{verbatim}
\usepackage[latin1]{inputenc}

\usepackage[mathscr]{eucal}

\usepackage{color}
\usepackage{bm}
\usepackage{bbm}
\usepackage{multirow}

\usepackage{latexsym, amsbsy, epsfig, enumerate}
\usepackage{mathabx}
\usepackage{mathsym}

\usepackage[normalem]{ulem}

\usepackage[colorlinks,citecolor=blue,urlcolor=blue,breaklinks]{hyperref} 

\definecolor{red2}{rgb}{0.9333333, 0, 0}
\definecolor{mediumblue}{rgb}{0, 0, 0.8039216}
\definecolor{magenta4}{rgb}{0.545098, 0, 0.545098}

\newcommand{\PP}{\mathsf{P}}

\newcommand{\suman}{\ensuremath{\sum_{i=1}^{n}}}

\newcommand{\E}{\ensuremath{\mathsf{E}}}
\newcommand{\Es}{\ensuremath{\mathsf{E}\,}}

\newcommand{\ind}{\ensuremath{\mathbb{I}}}
\newcommand{\Chat}{\ensuremath{\widehat{C}}}
\newcommand{\Fhat}{\ensuremath{\widehat{F}}}

\newcommand{\Ghat}{\ensuremath{\widehat{G}}}

\newcommand{\htilde}{\ensuremath{\widetilde{h}}\xspace}
\newcommand{\Jtilde}{\ensuremath{\widetilde{J}}\xspace}
\newcommand{\Jring}{\ensuremath{\mathring{J}}\xspace}
\newcommand{\muhat}{\ensuremath{\widehat{\mu}}}

\newcommand{\that}{\ensuremath{\widehat{t}}\xspace}
\newcommand{\ttilde}{\ensuremath{\widetilde{t}}\xspace}

\newcommand{\hateps}{\ensuremath{\widehat{\eps}}\xspace}
\newcommand{\epshat}{\ensuremath{\widehat{\eps}}}

\newcommand{\epsb}{\ensuremath{\boldsymbol{\eps}}}
\newcommand{\xib}{\ensuremath{\boldsymbol{\xi}}}

\newcommand{\hatG}{\ensuremath{\widehat{G}}}

\DeclareRobustCommand{\inPr}[1][n]{\ensuremath{\xrightarrow[#1\rightarrow\infty]{P}}} 

\newtheorem{theorem}{Theorem}
\newtheorem{corollary}{Corollary}
\newtheorem{lemma}{Lemma}

\theoremstyle{remark}
\newtheorem{remark}{Remark}

\DeclareRobustCommand{\tr}{^\mathsf{T}}
\renewcommand{\baselinestretch}{1.24}

\makeatletter
\def\singlespace{\def\baselinestretch{1}\@normalsize}

\def\boxit#1{\vbox{\hrule\hbox{\vrule\kern6pt
          \vbox{\kern6pt#1\kern6pt}\kern6pt\vrule}\hrule}}

\newtheorem{assumpT}{}

\newtheorem{assumpC}{}

\newtheorem{assumpZ}{}

\definecolor{green}{rgb}{0.03, 0.57, 0.3}

\begin{document}

\vspace*{-0.8 cm}

\noindent
\title[Generalized Hadamard differentiability of the copula mapping]
{Generalized Hadamard differentiability of the copula mapping and its 
applications}
\author{Natalie Neumeyer$^{\MakeLowercase{a}}$, Marek  Omelka$^{\MakeLowercase{b}}$  
}

\maketitle

\begin{center}
$^a$ 
Department of Mathematics, University of Hamburg, Bundesstrasse 55, 20146 Hamburg, Germany \\
$^b$
Department of Probability and Statistics,
Faculty of Mathematics and Physics,  
Charles University, 
Sokolovsk\'a 83,
186\,75 Praha 8, Czech Republic
\\

\end{center}
\vspace*{0.6 cm}

\begin{center}
\today
\end{center}

\date{\today}

\maketitle

\begin{abstract}
 We consider the copula mapping, which maps a joint cumulative distribution function to the corresponding copula. 
Its Hadamard differentiablity was shown in \cite{vaart_wellner}, \cite{fermanian-bernoulli-2004} and (under less strict assumptions) in \cite{bucher2013empirical}. 
This differentiability result has proved to be 
a powerful tool to show weak convergence of empirical copula processes in various settings using the functional delta method. 
We state a generalization of the Hadamard differentiability results 
that simplifies the derivations of asymptotic expansions and weak convergence of empirical copula processes in the presence of covariates. 
The usefulness  of this result is illustrated on several applications     
which include a multidimensional functional linear model, where the copula of the error vector describes the 
dependency between the components of the vector of observations, given the functional covariate. 
\end{abstract}

\section{Introduction}

Consider a $d$-dimensional random vector $\Yb = (Y_1, \ldots , Y_d)\tr$ 
with continuous marginal cumulative distribution functions~$F_{1},\dotsc,F_{d}$. Then by the famous Sklar's Theorem \citep[see e.g.][]{nelsen_2006} 
there exists a unique copula function~$C$ such that the joint cumulative 
distribution function~$F$ can be written as
\[
 F(y_1,\dotsc,y_d) = C\big(F_{1}(y_1),\dotsc, F_{d}(y_d)\big).  
\]
The interest in copulas in applications is rooted in the fact that the 
copula function~$C$ captures the dependence structure of~$\Yb$. 
From the beginning of the use of copulas in statistical modeling 
researchers were also interested in the empirical copula~$C_{n}$ which 
presents a natural nonparametric estimator of~$C$. Suppose you observe 
random vectors $\Yb_{1},\dotsc,\Yb_{n}$ such that each of the vector 
has the same cumulative distribution function (cdf)~$F$. Then the empirical 
copula is given by 
\begin{equation*} %
 C_{n}(u_1,\dotsc,u_d) = \Fhat_{n}\big(\Fhat_{1n}^{-1}(u_1),\dotsc, 
  \Fhat_{dn}^{-1}(u_d)\big), 
\end{equation*}
where $\Fhat_{n}$ is the empirical cdf (based on $\Yb_{1},\dotsc,\Yb_{n}$) and $\Fhat_{jn}^{-1}$ ($j \in \{1,\dotsc,d\}$) denotes the generalized inverse of the marginal empirical cdf $\Fhat_{jn}$. 

Provided that $\Yb_{1},\dotsc,\Yb_{n}$ are independent and identically 
distributed the asymptotic properties 
of~$C_n$ as a process in~$(u_1,\dotsc,u_d)$  has been studied already 
by \cite{gaenssler-stutte-1987} and then later on by 
\cite{fermanian-bernoulli-2004} and \cite{tsukahara_2005} among others. 
Nevertheless the weak convergence of the copula process 
$\mathbb{C}_{n} = \sqrt{n}(C_n - C)$ on $[0,1]^{d}$ under the present 
`standard assumptions' on~$C$ (see Assumption~\ref{assump: copula1}) was 
for the first time proved in \cite{segers_empirical_2012}. 

Already before the paper by \cite{segers_empirical_2012} some researchers 
\citep[see e.g. Section~3.9.4.4 of][]{vaart_wellner, fermanian-bernoulli-2004} considered 
the empirical copula function~$C_{n}$ as a mapping of the 
standard (multivariate) empirical process. To explore this approach 
in detail it is convenient to write 
the empirical copula process~$C_{n}$ as  
\begin{equation*} %
 C_{n}(u_1,\dotsc,u_d) = \Ghat_{n}\big(\Ghat_{1n}^{-1}(u_1),\dotsc, 
  \Ghat_{dn}^{-1}(u_d)\big), 
\end{equation*}
where 
\begin{equation} \label{eq: process Gn}
 \Ghat_{n}(u_1,\dotsc,u_d) = 
  \frac{1}{n} \suman \ind\big\{Y_{1i} \leq F_{1}^{-1}(u_1),\dotsc, Y_{di} \leq F_{d}^{-1}(u_d) \big\},  
\end{equation}
with the corresponding marginals 
\[
  \Ghat_{jn}(u) =   \frac{1}{n} \suman \ind\big\{Y_{ji} \leq F_{j}^{-1}(u)\big\},
\]
where $\Yb_{i}=(Y_{1i},\dotsc,Y_{di})\tr$, $i=1,\dotsc,n$.
Note that with this notation one can write the 
empirical copula as $ C_{n} = \Phi(\Ghat_{n})$, 
where $\Phi$ is the copula mapping 
which is for cdf~$H$ on $[0,1]^{d}$  defined as 
\begin{equation*} %
 \Phi: H \to H\big(H_{1}^{-1},\dotsc, H_{d}^{-1}\big),   
\end{equation*}
with $H_{j}^{-1}$ being the generalized inverse of the marginal cdf~$H_{j}$.  The advantage of this approach is that 
provided the copula mapping~$\Phi$ is appropriately 
Hadamard differentiable at~$C$ (with the derivative $\Phi'_{C}$) and the empirical process~$\sqrt{n}\,(\Ghat_{n}-C)$ 
defined in \eqref{eq: process Gn} converges weakly 
then by Theorem~3.9.4 of \cite{vaart_wellner} 
one gets 
\begin{equation} \label{vw theorem}
  \sqrt{n}(C_{n}- C) = \sqrt{n}\,\big(\Phi(\Ghat_{n}) -  \Phi(C)\big) 
  = \Phi'_{C}\big(\sqrt{n}(\Ghat_{n} -  C) \big) + o_{P}(1). 
\end{equation}
Thus with the help of functional delta method the asymptotic distribution of the empirical copula 
can be deduced simply from the asymptotic distribution of the process 
$\sqrt{n}\,(\Ghat_{n} -  C)$ plus the knowledge of $\Phi'_{C}$.

The needed differentiability result was proved by 
 \cite{bucher2013empirical} who showed that 
 under the same standard assumption~\ref{assump: copula1} 
 as in \cite{segers_empirical_2012} the mapping~$\Phi$ 
 is Hadamard differentiable at~$C$ tangentially to the 
 set of functions~$\mathcal{D}_0=\{h\in C([0,1]^d): h(1,\dots,1)=0\mbox{ and } h(\ub)=0\mbox{ if some of the components of }\ub \mbox{ are } 0 \}$
 with the derivative given by 
 \begin{equation} \label{eq: dPhi}
 \Phi'_{C}(h)(\ub) = h(\ub)-\sum_{j=1}^d C^{(j)}(\ub) h(\ub^{(j)}),     
 \end{equation}
where $C^{(j)} = \partial C/\partial u_{j}$,  $\ub = (u_1,\dotsc,u_d)$ and $\ub^{(j)}$ denotes the vector  whose all entries of $\ub$ except the $j$-th  are equal to~$1$, 
i.e. 
\begin{equation} \label{eq: ubj}
 \ub^{(j)} = (1,\dotsc,1,u_{j},1,\dotsc,1). 
\end{equation}

It should be stressed that this result on Hadamard differentiability 
of the copula functional simplifies the derivation of the asymptotic 
distribution of the empirical copula as it is sufficient only 
to prove the weak convergence of~$\sqrt{n}\,(\Ghat_{n}-C)$. This was 
utilized in \cite{bucher2013empirical} where the authors extended 
the results on weak convergence of the copula process $\mathbb{C}_{n}$ 
to situations when $\Yb_{1},\dotsc,\Yb_{n}$ are not independent but 
they follow for instance some mixing type conditions. 

\medskip 

The contribution of our paper is a generalization of the Hadamard differentiablity of the copula mapping in order to obtain weak convergence of the empirical copula process in situations where previous results are not applicable. This generalization was motivated 
by dealing with  empirical copula processes based on 
pseudo-observations/residuals of the form 
$\widehat{\eps}_{ji}=\widehat{t}_j(Y_{ji},\Xb_i)$, ($i=1,\dotsc,n$, $j\in \{1,\dotsc,d\}$), where the mapping $\widehat{t}_j$ depends on the observed sample. To explain our approach we first consider the special case of copula estimation based on residuals in a regression model (the general description of these problems will be given in Section~\ref{sec: applic to empir resid copula}). 

\subsection{Empirical copula based on residuals - motivation example}
\label{subsec motiv example}

Suppose we observe identically distributed random pairs $\binom{\Yb_{1}}{\Xb_{1}},\dotsc,\binom{\Yb_{n}}{\Xb_{n}}$ of a generic random pair $\binom{\Yb}{\Xb}$, 
where $\Yb_{i}=(Y_{1i},\dotsc,Y_{di})\tr$ 
and $\Xb_{i}$ is a (univariate or multivariate or even functional) covariate. 
Further suppose that the following homoscedastic 
regression models for each of the components of the marginal 
response hold
\begin{equation} \label{eq: homosc model}
 Y_{ji}=\mu_j(\Xb_i)+\varepsilon_{ji},\quad i=1,\dotsc,n, 
\end{equation}
where the random vector of 
innovations~$\epsb_{i}=(\eps_{1i},\dotsc,\eps_{di})\tr$ 
is independent of~$\Xb_{i}$ and has a continuous distribution. Let $\epsb$ be the 
generic random vector corresponding to $\epsb_{1},\dotsc,\epsb_{n}$.  
Denote $F_{\eps}$, $F_{j\eps}$ and $C$ respectively 
the joint cdf, the marginal cdf and copula of~$\epsb$. Note that 
the joint distribution of~$\Yb$ given $\Xb$ is  
thanks to Sklar's theorem and independence of $\epsb$ and $\Xb$ 
given by 
\begin{align*}
 \PP\big(Y_{1} \leq y_1, \dotsc, Y_{d} \leq y_d \,|\, \Xb = \xb \big) 
&= \PP\big(\eps_{1} \leq y_1 - \mu_{1}(\xb), \dotsc, \eps_{d} \leq y_d - \mu_{d}(\xb) \big) 
\\
&= C\Big(F_{1\eps}\big(y_1 - \mu_{1}(\xb)\big), \dotsc, F_{d\eps}\big(y_d - \mu_{d}(\xb)\big) \Big)
\end{align*}
with conditional marginals $\PP(Y_{j} \leq y_j \,|\, \Xb = \xb ) =F_{j\eps}(y_j - \mu_{j}(\xb))$, $j\in\{1,\dotsc,d\}$.
Thus the copula function~$C$ of $\epsb$ captures 
the conditional dependence structure of~$\Yb$ given~$\Xb$. 

Let $\muhat_{j}$ be a regression estimator of the location function $\mu_j$ and 
\begin{equation} \label{eq: homosc resid}
\widehat{t}_j(Y_{ji},\Xb_i)= \epshat_{ji}=Y_{ji} - \muhat_j(\Xb_i), \quad i =1,\dotsc,n,  \, j\in\{1,\dotsc,d\},    
\end{equation}
be the corresponding residuals. Then the straightforward 
empirical copula estimator of~$C$ is 
\begin{equation} \label{eq: definition C_n}
\Chat_{n}(\ub) 
  = \Fhat_{n\epshat}\big(\Fhat_{1\epshat}^{-1}(u_1),\dotsc, 
  \Fhat_{d\epshat}^{-1}(u_d)\big),  \quad \ub \in [0,1]^d, 
\end{equation}
where $\Fhat_{n\epshat}$ is an empirical cdf of the residuals 
$\widehat{\epsb}_{i} = (\epshat_{1i},\dotsc,\epshat_{di})\tr$, $i=1,\dotsc,n$, 
and $\Fhat_{1\epshat}^{-1},\dotsc,\Fhat_{d\epshat}^{-1}$ be the corresponding generalized inverses.

Assume for a moment that for each $j$ the cdf $F_{j\eps}$ is \textbf{strictly increasing} (on a set where $\epshat_{j1},\dotsc,\epshat_{jn}$ take values). Then similarly as above one can rewrite $\Chat_{n}$ as 
\begin{equation} \label{eq: empir resid copula via Gn}
\Chat_{n}(\ub) 
  = \Ghat_{n\epshat}\big(\Ghat_{1\epshat}^{-1}(u_1),\dotsc, 
  \Ghat_{d\epshat}^{-1}(u_d)\big),  \quad \ub \in [0,1]^d,  
\end{equation}
where 
\begin{equation} \label{eq: Ghat epshat}
 \Ghat_{n\epshat}(\ub) = \frac{1}{n} \suman 
 \ind\big\{ \epshat_{1i} \leq F_{1\eps}^{-1}(u_1), 
 \dots,  \epshat_{di} \leq F_{d\eps}^{-1}(u_d) \big\},  
\end{equation}
and $\Ghat_{1\epshat}^{-1},\dotsc,\Ghat_{d\epshat}^{-1}$ be the corresponding generalized inverses 
of marginals
\[
 \Ghat_{j\epshat}(u) = \frac{1}{n} 
  \suman \ind\big\{ \epshat_{ji} \leq F_{j\eps}^{-1}(u) \big\}, \quad j\in\{1,\dotsc,d\}.  
\]

Note that $\Chat_n$ mimics the ideal (`oracle') empirical copula
\[
 C_n^{(or)}(\ub) 
  = \Ghat_{n\eps}\big(\Ghat_{1\eps}^{-1}(u_1),\dotsc, 
  \Ghat_{d\eps}^{-1}(u_d)\big),  \quad \ub \in [0,1]^d,  
\]
that would be based on the empirical cdf of the 
true (but unobserved) errors 
\begin{equation} \label{eq: Ghat eps}
 \Ghat_{n\eps}(\ub) = \frac{1}{n} \suman 
 \ind\big\{ \eps_{1i} \leq F_{1\eps}^{-1}(u_1), 
 \dots,  \eps_{di} \leq F_{d\eps}^{-1}(u_d) \big\}.   
\end{equation}
Nevertheless one can often (i.e. under appropriate regularity assumptions) prove that 
\begin{equation} \label{eq: equiv of Cn and oracle Cn}
  \sqrt{n}\,\big(\Chat_{n} - C_{n}^{(or)}\big) = o_{P}(1). 
\end{equation}
Thus the asymptotic distribution of $\Chat_{n}$ is the same as 
$C_{n}^{(or)}$. Note that this is rather surprising as one would intuitively expect that the uncertainity in estimation of $\mu_{j}$ 
should propagate to the asymptotic distribution of $\Chat_{n}$ (which is the case for $\Ghat_{n\hateps}$, see (\ref{eq: reprentation of Geps}) below).  

To show \eqref{eq: equiv of Cn and oracle Cn} one has basically two possibilities. The more elegant 
approach is to note that $\Chat_{n} = \Phi(\Ghat_{n\epshat})$ 
and to use the Hadamard differentiability proved in 
\cite{bucher2013empirical}. This approach was used for instance 
in \cite{noh_copula_nts_2019}. The disadvantage of this approach is that 
 weak convergence of the process 
$\mathbb{G}_{n\epshat}= \sqrt{n}(\Ghat_{n\epshat} - C)$ is needed 
which requires more strict assumptions or even may be rather  
problematic in some applications (see Section~\ref{sec: applications}). 
The  less elegant approach is to deal directly 
with the estimator $\Chat_{n}$ for given specific models  
as for instance in \cite{ogv_SCAJS_2015} or \cite{portier2018weak}. This 
approach does not require the weak convergence 
of $\mathbb{G}_{n\epshat}$. But on the other hand 
it is more cumbersome and technical as 
one needs to deal with the process 
\[
\mathbb{C}_{n}(\ub) = \sqrt{n}\,\Big[\Ghat_{n\epshat}\big(\Ghat_{1\epshat}^{-1}(u_1),\dotsc, 
  \Ghat_{d\epshat}^{-1}(u_d)\big) - C(\ub)\Big], \quad \ub \in [0,1]^{d}  , 
\]
instead of the simpler process $\mathbb{G}_{n\epshat}$. 
That is why the results in the literature are usually tied 
to the model and to the method of estimation of the effect 
of the covariate on the marginals.

The way how Hadamard differentiability is used to prove 
the surprising result~\eqref{eq: equiv of Cn and oracle Cn} 
can be summarized follows. First one shows that
uniformly in $\ub \in [0,1]^{d}$ 
\begin{equation} \label{eq: reprentation of Geps}
\Ghat_{n\epshat}(\ub) = C(\ub) + n^{-1/2}\mathbb{A}_n(\ub) 
 + n^{-1/2}\mathbb{B}_n(\ub),    
\end{equation}
where 
\begin{align} \label{eq: An typical}
 \mathbb{A}_n(\ub) &= \sqrt{n}\,\big(\Ghat_{n\eps}(\ub) - C(\ub)\big) + o_P(1), 
 \\
\label{eq: Bn typical}
\mathbb{B}_n(\ub) &=\sum_{j=1}^d C^{(j)}(\ub)\, \ZZZ_{jn}(u_j),  
\end{align}
with $\Ghat_{n\eps}$ introduced in~\eqref{eq: Ghat eps} and 
$\ZZZ_{jn}$ that is in model~\eqref{eq: homosc model} given by 
\begin{equation*} %
  \ZZZ_{jn}(u) =  \sqrt{n}\, \E_{\Xb} 
 \big [F_{j\eps}\big(F_{j\eps}^{-1}(u)+\widehat{\mu}_j(\Xb)-\mu_j(\Xb)\big)-u], 
 \quad j \in \{1,\dotsc,d\}.   
\end{equation*}
Here and throughout $\E_{\Xb}$ denotes the expectation over $\Xb$, which is the generic covariate, independent of the sample, keeping all other random variables fixed. 

Further one shows that the (joint) process 
\begin{equation} \label{eq: joint process A_n Z_j}
 \big\{\big(\AAA_{n}(\ub), \mathbb{Z}_{1n}(u_1),\dotsc,\mathbb{Z}_{dn}(u_d)\big); 
  \ub \in [0,1]^{d}\, \big\}
\end{equation}
converges weakly (jointly). Thus also the processes $\AAA_{n}$ and 
$\mathbb{B}_{n}$ converge weakly (jointly). 
So one can apply the Hadamard differentiability as in (\ref{vw theorem}) together with the 
representation~\eqref{eq: reprentation of Geps} to get 
\begin{equation} \label{eq: approx of Cn with HD prelim}
 \sqrt{n}\,\big(\Chat_{n} - C \big) 
  =  \sqrt{n}\, \big(\Phi(\Ghat_{n\epshat})-\Phi(C)\big)
   = \Phi'_{C} \big(\AAA_{n}\big) + \Phi'_{C} \big(\mathbb{B}_{n}\big) 
   + o_{P}(1),
\end{equation}
where we used the linearity of $\Phi'_{C}$. Now note that 
it follows from the specific structure of $\mathbb{B}_n$ in (\ref{eq: Bn typical}),  the Hadamard derivative in (\ref{eq: dPhi}) and $C^{(j)}(\ub^{(j)})=1$ that 
\begin{align*} 
 \Phi'_{C} \big(\mathbb{B}_{n}\big)(\ub) 
 &= \mathbb{B}_{n}(\ub)-\sum_{k=1}^dC^{(k)}(\ub)\,\mathbb{B}_{n}(\ub^{(k)})\\
 &=   
 \sum_{j=1}^d C^{(j)}(\ub)\, \ZZZ_{jn}(u_j)
 -\sum_{k=1}^dC^{(k)}(\ub)\sum_{j=1}^d C^{(j)}(\ub^{(k)})\, \ZZZ_{jn}(u_{j}^{(k)})
 \\
 &= -\sum_{k=1}^dC^{(k)}(\ub)\sum_{j=1\atop j\neq k}^d C^{(j)}(\ub^{(k)})\, \ZZZ_{jn}(1), 
\end{align*}
where $u_{j}^{(k)}$ stands for the $j$-th coordinate of $\ub^{(k)}$. 
Thus with the help of $0\leq C^{(j)}(\ub)\leq 1$ 
for all $\ub \in [0,1]^{d}$ 
and the fact that the limiting processes $\ZZZ_{1},\dotsc,\ZZZ_{d}$ 
corresponding to $\ZZZ_{1n},\dotsc,\ZZZ_{dn}$ 
typically satisfy $\PP\big(\ZZZ_{j}(1)= 0\big) = 1$ one can obtain
\begin{equation} \label{eq: dPhi at Bn diminishes}
 \sup_{\ub \in [0,1]^{d}}\big|\Phi'_{C} \big(\mathbb{B}_{n}\big)(\ub)\big| 
  \leq (d-1) \sum_{j=1}^{d} \big|\ZZZ_{jn}(1)\big| = o_{P}(1).     
\end{equation}
Now combining 
\eqref{eq: approx of Cn with HD prelim} and \eqref{eq: dPhi at Bn diminishes} 
we get the approximation 
\begin{equation} \label{eq: approx of Cn with HD}
 \sqrt{n}\,\big(\Chat_{n} - C \big) 
  = \Phi'_{C} \big(\AAA_{n}\big) 
   + o_{P}(1),       
\end{equation}
where $\mathbb{B}_{n}$ is not present on the right-hand side.  
So it remains to note that the right-hand side of the last equation 
coincides with the asymptotic representation of the empirical 
copula based on unobserved errors~$C_{n}^{(or)}$ (see (\ref{vw theorem}) and (\ref{eq: An typical})) which 
implies~\eqref{eq: equiv of Cn and oracle Cn}.

\medskip 

The aim of this paper is to introduce two useful modifications of the Hadamard differentiability result of \cite{bucher2013empirical} that 
require milder properties of the process~$\mathbb{B}_{n}$ than 
the weak convergence, but still yield the result (\ref{eq: approx of Cn with HD}). This will present not only a technique that will 
simplify the proofs. It will be also useful in situations where 
one is either not able to prove the weak convergence of~$\mathbb{B}_{n}$ 
(i.e.\ typically the weak convergence of the processes~$\ZZZ_{1n},\dotsc,\ZZZ_{dn}$) or this convergence 
simply does not hold. See Section~\ref{sec: applications} for such applications. 
Finally it is worth noting that it is rather intuitive not to require 
the weak convergence of $\mathbb{B}_{n}$ as this process is not 
present in the final asymptotic representation \eqref{eq: approx of Cn with HD}.

\medskip 

The paper is organized as follows. In Section~\ref{sec: HD} 
we state the new Hadamard diffentiability results. In 
Section~\ref{sec: applic to empir resid copula} we discuss their 
use to empirical copulas based on pseudo-observations in general. In 
Section~\ref{sec: applications} 
we give some more specific applications where the results on the 
weak convergence of the process~$\mathbb{B}_{n}$ are either not available or require 
more stringent assumptions. Some conclusions and further discussions can be found in Section~\ref{sec: discussion}. All the proofs are given in
Appendices.

\section{Hadamard differentiability of copula 
functional~\texorpdfstring{$\Phi$}{Phi}}
\label{sec: HD}

Similarly as in \cite{bucher2013empirical} let $\mathcal{D}_{\Phi}$ 
be the set of distribution functions on~$[0,1]^{d}$ whose marginal cdfs 
satisfy $H_{j}(0) = 0$ and $H_{j}(1)=1$ for each $j \in \{1,\dotsc,d\}$. 
Further let 
\begin{equation} \label{eq: generalised inverse}
 H_{j}^{-1}(u) = \inf\big\{ v \in [0,1]: H_{j}(v) \geq u\big\}, \quad u \in [0,1], 
\end{equation}
be the corresponding generalized inverse function. 

\medskip

Provided that~\eqref{eq: reprentation of Geps}  holds we will use the weak convergence of $\mathbb{A}_n$, but not of $\mathbb{B}_n$. The technique to prove~\eqref{eq: approx of Cn with HD} 
can be summarized as follows. For each $\eta > 0$ we find 
sequences of sets of functions, say $\mathcal{A}_{n}$ and 
$\mathcal{B}_{n}$, such that on one hand
\begin{equation*} %
 \liminf_{n \to \infty} \PP\big(\AAA_{n} \in \mathcal{A}_{n},  
 \mathbb{B}_{n} \in \mathcal{B}_{n}\big)
 \geq 1 - \eta.    
\end{equation*}
But at the same time uniformly in $h \in \mathcal{A}_{n}$, $\htilde \in \mathcal{B}_{n}$ 
(provided that $(C + \tfrac{h}{\sqrt{n}} + \tfrac{\htilde}{\sqrt{n}}) \in \mathcal{D}_{\Phi}$)   
\begin{equation*} %
  \sup_{\ub \in I_{n}} 
  \bigg|\sqrt{n}\, 
  \big(\Phi(C + \tfrac{h}{\sqrt{n}} + \tfrac{\htilde}{\sqrt{n}})(\ub) - \Phi(C)(\ub) \big)  - \Phi'_{C}(h)(\ub) \bigg| 
 \ntoinfty 0,       
\end{equation*}
where $I_{n}$ is either $[0,1]^{d}$ (see Theorem~\ref{thm HD of copula fctional}) or 
an appropriate sequence of subsets of $[0,1]^{d}$ 
(see Theorem~\ref{thm HD of copula fctional second}).

\medskip 

In what follows we will keep in mind the representation~\eqref{eq: reprentation of Geps} and find appropriate sets of functions 
for processes~$\AAA_{n}$ and ~$\mathbb{B}_{n}$. 

\subsection*{Sets of functions for the process 
\texorpdfstring{$\AAA_{n}$}{An}} 

Let $\ell^{\infty}([0,1]^d)$ be the set of bounded functions on~$[0,1]^{d}$ and introduce 
\begin{multline*} %
 \mathcal{L} = \big\{ h \in \ell^{\infty}([0,1]^d): h(1,\dotsc,1)=0 
 \\ \text{ and } 
h(\ub) = 0 \text{ if some of the components of } \ub \text{ are equal to 0} \big\}.   
\end{multline*}

In our applications we will consider models for which the process $\AAA_{n}$ 
converges in distribution to a limiting process~$\mathbb{A}$  
that satisfies 
\[
 \PP\big(\mathbb{A} \in \mathcal{L} \cap \mathcal{C}([0,1]^d) \big) = 1.   
\]
Note that this includes not only the i.i.d.\ setting but also various 
weak dependence concepts for strictly stationary sequences 
\citep[see the discussion below Condition~2.1 in][]{bucher2013empirical}.

Thus the above weak convergence of the process~$\AAA_{n}$ for each $\eta > 0$ 
implies that one can find a sequence of sets of functions~$\mathcal{A}_{n}$ which fulfills the following conditions
such that 
\begin{equation} \label{eq: An in Kn}
 \liminf_{n \to \infty} \PP\big(\AAA_{n} \in \mathcal{A}_{n}\big) \geq 1 - \eta.   
\end{equation}
The functions in~$\mathcal{A}_{n}$ are asymptotically 
uniformly equi-continuous, i.e. for each $\varrho >0$ there exists $\delta >0$ 
such that for all sufficiently large~$n$ 
\begin{equation} \label{eq: as unif equi cont}
 \sup_{h \in \mathcal{A}_{n}} \sup_{\|\ub - \vb\| < \delta} 
 \big|h(\ub) - h(\vb) \big| < \varrho.  
\end{equation}
Moreover the functions in~$\mathcal{A}_{n}$ are uniformly bounded and asymptotically close to the set~$\mathcal{L}$, 
i.e. 
\begin{equation} \label{eq: as bound and close to L}
\sup_{n \in \NN}\,
\sup_{h \in \mathcal{A}_{n}} \sup_{\ub \in [0,1]^{d}} |h(\ub)| < \infty 
 \quad  \text{and} \quad 
 \sup_{h \in \mathcal{A}_{n}}\, \inf_{g \in \mathcal{L}}\,\sup_{\ub \in [0,1]^{d}} |h(\ub)-g(\ub)|  
 \ntoinfty 0. 
\end{equation}

\begin{remark} \label{rem: asympt diminish at borders}
 For $j \in \{1,\dotsc,d\}$ let $h_{j}(u) = h(\ub^{(j)})$, where $\ub^{(j)}$ 
 was introduced in \eqref{eq: ubj}.  Note that 
 \eqref{eq: as unif equi cont} and \eqref{eq: as bound and close to L} 
 imply that for each $\varrho > 0$
 there exists $\delta >0$ such that for all sufficiently large~$n$ 
\begin{equation*}
  \sup_{h \in \mathcal{A}_{n}} 
     \sup_{u \in [0,\delta] \cup [1-\delta,1]} |h_{j}(u)|  < \varrho. 
\end{equation*}
\end{remark}

\subsection{Assumptions on the copula function}
In what follows we will consider the following two versions 
of the assumptions on copula~$C$.

\begin{assumpC} \label{assump: copula1}
For each $j \in \{1,\dotsc,d\}$ the first-order partial derivative 
$C^{(j)} = \partial C/\partial u_{j}$ exists and is continuous on the 
set $\{\ub \in [0,1]^{d}: 0 < u_{j} < 1 \}$. 
\end{assumpC}

\begin{assumpC} \label{assump: copula2}
There is $\beta \in [0, \frac{1}{2}]$ such that for 
each $j,k \in \{1,\dotsc,d\}$ the second order 
partial derivative $C^{(j,k)} = \partial^{2} C/(\partial u_{j} \partial u_{k})$ 
exists and satisfies  
\begin{equation*}
  C^{(j,k)}(\ub) 
  =O\Big(\tfrac{1}{u_{j}^{\beta}(1-u_{j})^{\beta}
 \,u_{k}^{\beta}(1-u_{k})^{\beta}}\Big), 
\qquad \ub = (u_{1},\dotsc,u_{d}) \in (0,1)^{d}. 
\end{equation*}
\end{assumpC}

Note that assumption \ref{assump: copula1} does not 
imply the existence of the first-order 
partial derivative $C^{(j)}$ on the complement of the set $\{\ub \in [0,1]^{d}: 0 < u_{j} < 1 \}$. 
As it will be evident later (see the definition of the sets of functions 
$\mathcal{B}$ and $\mathcal{B}_{n}^{\alpha}$ in \eqref{eq: tilde hn} and 
\eqref{eq: tilde hn for copula2} below) it is irrelevant how $C^{(j)}$ is defined 
on that complement. Nevertheless, for simplicity of notation 
it is convenient that $C^{(j)}$ is defined on $[0,1]^{d}$. To have that one can define  
$C^{(j)}$ for instance as zero in points where $C^{(j)}$ does not exist. 
Similarly for assumption \ref{assump: copula2} which does not say anything 
about the existence of $C^{(j)}$ even on the complement of $(0,1)^{d}$. 

It is worth noting that assumption~\ref{assump: copula1} is the standard 
copula assumption that was introduced in \cite{segers_empirical_2012}. 
It is also the assumption under which the Hadamard differentiability 
result was proved in \cite{bucher2013empirical}. Roughly speaking 
the corresponding differentiability result (see 
Section~\ref{subsec HD of copula assump1}) is useful when 
the effect of the covariate on the marginals can be removed 
in $\sqrt{n}$-rate. That is for instance in our motivating 
example in Section~\ref{subsec motiv example} when one (rightly) assumes 
a parametric model for~$\mu_{j}$. 

On the other hand assumption~\ref{assump: copula2} is more strict 
but for $\beta = \tfrac{1}{2}$  still satisfied for many common copulas \citep[see e.g.][]{ogv_as_2009}. This more strict assumption 
on the copula function and the corresponding differentiability 
result (see Section~\ref{subsec HD of copula assump2}) can be used to compensate 
the fact that one is able to remove the effect of the covariate on the marginals only 
at a slower than~$\sqrt{n}$-rate.

\subsection{Hadamard differentiability of copula mapping under assumption~\ref{assump: copula1}}
\label{subsec HD of copula assump1} 

In what follows we introduce a set of functions 
for the process~$\mathbb{B}_{n}$ that is of the 
form~\eqref{eq: Bn typical}.  To do that 
first for $M \in (0, \infty)$ introduce the set of bounded functions  
\begin{equation} \label{eq: set HM}
 \widetilde{\mathcal{H}}_{M} = \big\{\, \widetilde{h} \in \ell^{\infty}([0,1]) :  
   \sup_{u \in [0,1]} |\widetilde{h}(u)| \leq M  
  \big\}.  
\end{equation}

Further let 
$r$ be a bounded non-negative function such that 
\begin{equation} \label{eq: r is diminishing at borders}
 \lim_{u \to 0_{+}} r(u) = 0 = \lim_{u \to 1_{-}} r(u). 
\end{equation}
Now 
consider the set of bounded univariate functions 
\begin{equation*} %
 \mathcal{B}_{1}
 = \big\{\widetilde{h}\in \widetilde{\mathcal{H}}_{M}: |\widetilde{h}(u)| \leq r(u);  
 \, \forall u \in \big[0,1] \big\}   
\end{equation*}
and the corresponding set of multivariate functions 
\begin{equation} \label{eq: tilde hn}
\mathcal{B} = \Big\{ \widetilde{h} \in \ell^{\infty}([0,1]^d): 
 \widetilde{h}(\ub) = \sum_{j=1}^{d} C^{(j)}(\ub)\,\htilde_{j}(u_j);  
\text{ where }   \htilde_{j} \in \mathcal{B}_{1}, \, 
\forall j \in \{1,\dotsc,d\} 
\Big\}.
\end{equation}

Now for a given  sequence of positive constants $\{t_{n}\}$ going 
to zero introduce  a sequence  $\{\mathcal{D}_{n}\}$ of subsets of 
$\mathcal{D}_{\Phi}$ such that the marginals~$H_{j}$ 
have jumps of height at most $o(t_{n})$ 
\begin{equation} \label{eq: Dn}
 \sup_{H \in \mathcal{D}_{n}} \max_{j \in \{1,\dotsc,d\}} \sup_{u \in (0,1]} 
 \big|H_{j}(u)- H_{j}(u_{-})\big| = o(t_{n}). 
\end{equation}

Finally introduce 
\[
 \mathcal{D}_{n}^{(1)} = \big\{H \in \mathcal{D}_{n}:\, H = C + t_{n}\,h + t_{n}\,\widetilde{h}, 
 \text{ where } h \in \mathcal{A}_{n}, \htilde \in \mathcal{B}  \big\}. 
\]

Now we are ready to formulate the main result of this section that 
can be considered as a generalization of Theorem~2.4 of \cite{bucher2013empirical}. 

\begin{theorem} \label{thm HD of copula fctional} 
Let the set of functions $\mathcal{D}_{n}^{(1)}$ be 
as explained above. Further let $C$ satisfy \ref{assump: copula1}. Then 
 \begin{equation*}
\sup_{C_{n}^{h,\htilde} \in \mathcal{D}_{n}^{(1)}}\, 
\sup_{\ub \in [0,1]^{d}}  
\bigg|\frac{\Phi(C + t_{n}\,h + t_{n}\,\widetilde{h})(\ub) 
- C(\ub)}{t_{n}} 
  - h(\ub) + \sum_{j=1}^{d} C^{(j)}(\ub)\,h(\ub^{(j)})
\bigg| \ntoinfty 0, 
 \end{equation*}
where $C_{n}^{h,\htilde} = C + t_{n}\,h + t_{n}\,\widetilde{h}$.   
\end{theorem}

\subsection{Hadamard differentiability of copula mapping under assumption~\ref{assump: copula2}} 
\label{subsec HD of copula assump2}
Note that in Theorem~\ref{thm HD of copula fctional} both perturbations  
of the copula function~$C$ presented by $h$ and $\htilde$ 
have the same rates $t_{n}$. In what follows we allow that 
the perturbation corresponding to $\htilde$ converges 
to zero at a rate $\ttilde_{n}$ which is slower than $t_{n}$. This will be useful 
when the processes $\ZZZ_{jn}$ in~\eqref{eq: Bn typical} 
are not bounded in probability. 

Of course there is some price to be paid which depends on the actual 
rate of $\ttilde_{n}$ that we want to allow. This price is paid 
partly by the more severe assumption on the copula~$C$ and 
partly by assuming a finer behaviour of $\htilde$ when one 
is close to zero or one. A part of the price is also that the 
result will not hold uniformly on $[0,1]^{d}$ but only at 
an increasing sequence of subsets of $[0,1]^{d}$.  

\smallskip 

For $\epsilon > 0$ and $\vartheta \in (0,1]$ denote 
\[
 \tilde{I}_{1n}(\epsilon) = 
\big[\epsilon\, t_{n}^{\vartheta}, 1 -  \epsilon\, t_{n}^{\vartheta}\big] ,  
\quad 
\tilde{I}_{n}(\epsilon) = \big[\epsilon\, t_{n}^{\vartheta}, 1 -  \epsilon\, t_{n}^{\vartheta}\big]^{d}. 
\] 
Further for $\alpha > 0$ and $\epsilon > 0$ fixed introduce the sets of functions 
\begin{equation} \label{eq: set B1alpha}
 \mathcal{B}_{1n}^{\alpha}
 = \bigg\{\widetilde{h}\in \widetilde{\mathcal{H}}_{M}:  
 \ 
\sup_{u \in \tilde{I}_{1n}(\epsilon)}  
\tfrac{|\htilde(u)|}{u^{\alpha}(1-u)^{\alpha}} 
  < 2\,M 
\bigg\}
\end{equation}
and 
\begin{equation} \label{eq: tilde hn for copula2} 
\mathcal{B}_{n}^{\alpha} = \Big\{\widetilde{h} \in {\ell^{\infty}([0,1]^d)}: 
 \widetilde{h}(\ub) = \sum_{j=1}^{d} C^{(j)}(\ub)\,\htilde_{j}(u_j);  
 \text{ where }   \htilde_{j} \in \mathcal{B}_{1n}^{\alpha}, \, 
\forall j \in \{1,\dotsc,d\} 
\Big\}.
\end{equation}

\medskip

Finally for sequences of constants $\{t_{n}\}$ and $\{\ttilde_{n}\}$  going to zero 
introduce 
\[
 \mathcal{D}_{n}^{(2)} = \big\{H \in \mathcal{D}_{n}:\, H = C + t_{n}\,h + \ttilde_{n}\,\htilde, 
 \text{ where } h \in \mathcal{A}_{n}, \htilde \in \mathcal{B}_{n}^{\alpha}  \big\} 
\]
and $\mathcal{D}_{n}$ was introduced in \eqref{eq: Dn}.

\begin{theorem} \label{thm HD of copula fctional second} 
Assume that \ref{assump: copula2} holds and 
$\alpha \geq 0$, 
$\gamma \in [0,1/4)$ and $\vartheta \in (0,1]$ be 
constants such that 
\begin{equation} \label{eq: assump on gamma and vartheta}
 \gamma \geq \tfrac{(\beta - \alpha)_{+}}{4(1-\beta)} 
 \quad \text{and} \quad 
 \vartheta = \min\big\{\tfrac{1+4 \gamma}{2(1-\alpha)}, 1 \big\}. 
\end{equation}
Further let  $\mathcal{D}_{n}^{(2)}$ be as explained above with 
$\ttilde_{n} = o\big(t_{n}^{1/2+2\gamma}\big)$. 
Then for each $\epsilon > 0$ and $M \in (0, \infty)$
 \begin{equation*} %
\sup_{C_{n}^{h,\htilde} \in \mathcal{D}_{n}^{(2)}}\, 
\sup_{\ub \in \tilde{I}_{n}(\epsilon)}  
\bigg|\frac{\Phi(C + t_{n}\,h 
 + \ttilde_{n}\,\widetilde{h})(\ub) 
- C(\ub)}{t_{n}} 
  - h(\ub) + \sum_{j=1}^{d} C^{(j)}(\ub)\,h(\ub^{(j)})
\bigg| \ntoinfty 0, 
 \end{equation*}
where $C_{n}^{h,\htilde} = C + t_{n}\,h + \ttilde_{n}\,\widetilde{h}$  and 
$\ub^{(j)}$ is given in \eqref{eq: ubj}. 
\end{theorem}
\section{Application to empirical copulas based on pseudo-observations} 
\label{sec: applic to empir resid copula}

In this section we generalize the model used in 
Section~\ref{subsec motiv example} and show how the results 
of the previous section can be applied.

Suppose we observe identically (but not necessarily independently) distributed random pairs $\binom{\Yb_{1}}{\Xb_{1}},\dotsc,\binom{\Yb_{n}}{\Xb_{n}}$ of a generic random pair 
$\binom{\Yb}{\Xb}$, where $\Yb_{i}=(Y_{1i},\dotsc,Y_{di})\tr$ 
and $\Xb_{i}$ is a covariate ($q$-dimensional or even functional). 
Often we are interested in the conditional distribution of $\Yb$ given the value 
of the covariate. To simplify the situation it is often assumed 
that $\Xb$ affects only the marginal 
distributions of~$Y_j\ (j \in \{1,\dotsc,d\})$,  but does not affect the dependence structure of~$\Yb$. 
More formally, it is assumed that there exists a copula $C$ such that the joint conditional distribution 
of $\Yb$ given $\Xb = \xb$ can be for all $\xb \in S_{\Xb}$ (the support 
of $\Xb$) written as 
\begin{equation*}
F_{\xb}(y_{1}, \ldots, y_{d}) = \PP (Y_{1} \leq y_{1}, \ldots, Y_{d} \leq y_{d} \mid \Xb = \xb ) = C \big(F_{1\xb}(y_1),\ldots,F_{d\xb}(y_d)\big)
\end{equation*}
where $F_{j\xb}(y_j)=\PP (Y_{j} \leq y_{j}\mid \Xb = \xb)$, $j \in \{1,\dotsc,d\}$.

Let $t_{j}(y;\xb)$ be a real valued function defined on 
$S_{Y_j}\times S_{\Xb}$ (where $S_{Y_j}$ is the support of $Y_j$)
such that the random variable 
$\eps_{j} = t_{j}(Y_{j};\Xb)$ is independent 
of~$\Xb$ and has a continuous distribution. Moreover assume 
that for each $\xb \in S_{\Xb}$ 
the function $t_{j}(\cdot;\xb)$ is increasing in the first argument. 
Then the copula $C$ is the copula of the  random 
vector $(\eps_{1},\dotsc,\eps_{d})\tr$. 

Note that one can always take $t_{j}(y;\xb) = F_{j\xb}(y)$. 
Nevertheless simpler functions can be available depending on the assumptions about the effect of the covariate 
on the marginal distributions. For instance in 
model~\eqref{eq: homosc model} one can use simply   
$t_{j}(y;\xb) = y - \mu_{j}(\xb)$.

\subsection{Empirical copula estimation}

Note that ideally one would base the estimate of~$C$ on random variables
\begin{equation*} %
  \eps_{ji} = t_{j}(Y_{ji}; \Xb_{i}).  
\end{equation*}
As the transformation $t_{j}$ is usually 
in practice unknown, then one needs to work with 
the `estimates' of $\eps_{ji}$ (`pseudo-observations')  
\begin{equation}  \label{eq: resid general}
 \widehat{\eps}_{ji} = \that_{j}(Y_{ji}; \Xb_{i}),  
 \quad i=1,\dotsc,n; \ j \in \{1,\dotsc,d\},  
\end{equation}
where for instance in model~\eqref{eq: homosc model} 
the pseudo-observations are the residuals as in \eqref{eq: homosc resid}. 
The empirical copula (based on estimated pseudo-observations) is then 
defined analogously as in Section~\ref{subsec motiv example}. 

Provided that for each $j$ the cdf~$F_{j\eps}$ of $\eps_{j}=t_{j}(Y_j;\Xb)$ (introduced above) is strictly increasing one can 
rewrite the empirical copula as \eqref{eq: empir resid copula via Gn} 
using function $\Ghat_{n\epshat}$ introduced in \eqref{eq: Ghat epshat}.  
More generally if $F_{j\eps}$ is not strictly increasing but it is continuous 
then one can always find  
a sequence of cdf's, say $\{\widetilde{F}_{nj\eps}\}$ such that $\widetilde{F}_{nj\eps}$ 
is strictly increasing for each $n$ and at the same time
\begin{equation} \label{eq: tilde Fnjeps close to Fjeps}
 \sup_{u \in [0,1]}\big|F_{j\eps}\big(\widetilde{F}_{nj\eps}^{-1}(u)\big) - u \big| = o\big(\tfrac{1}{\sqrt{n}}\big). 
\end{equation}
Now $\Ghat_{n\epshat}$ will not be defined as in \eqref{eq: Ghat epshat} but rather as 
\begin{equation*}  %
 \Ghat_{n\epshat}(\ub) = \frac{1}{n} \suman 
 \ind\big\{ \epshat_{1i} \leq \widetilde{F}_{n1\eps}^{-1}(u_1), 
 \dots,  \epshat_{di} \leq \widetilde{F}_{nd\eps}^{-1}(u_d) \big\}.   
\end{equation*}
Then \eqref{eq: empir resid copula via Gn} holds even if some of the cdfs 
$F_{1\eps},\dotsc,F_{d\eps}$ are not strictly increasing. 

\medskip

\begin{remark}
 Note that as we are in the conditional copula settings 
 we need to use the word `pseudo-observation'  
 in the broader sense than is often used in the literature about copulas.  
 In view of \cite{ghoudi-remillard-1998} 
 our pseudo-observations are functions of the observed $Y_{ji}$ 
 and the estimated conditional law~$F_{j\xb}$. 
\end{remark}

First we formulate generic assumptions that need to be 
verified for the given marginal models and methods of estimation. 

\subsection{Generic assumptions}

In what follows let $P$ stand for the measure of the 
random vector $\binom{\Yb}{\Xb}$ and $P_{n}$ be the corresponding 
empirical measure based on $\binom{\Yb_{1}}{\Xb_{1}},\dotsc,\binom{\Yb_{n}}{\Xb_{n}}$. 

Further  let $\mathcal{T}_{j}$ 
be a set of real functions  of the form $t_{j}(y; \xb)$ 
defined on $S_{Y_{j}} \times S_{\Xb} \to \RR$  that are for each fixed 
$\xb \in S_{\Xb}$ increasing in~$y$.   

While the first assumption justifies that the empirical 
process of pseudo-observations is asymptotically uniformly 
equicontinuous in probability \citep[see e.g. p.~38 of][]{vaart_wellner}, 
the second assumption ensures that pseudo-observations are 
consistent estimators of unobserved $\eps_{ji}$.

\vspace*{-2mm}

\begin{assumpT} \label{assump: donsker}
Suppose that for each $j \in \{1,\dotsc,d\}$ there exists   
a set of real functions $\mathcal{T}_{j}$ 
such that 
\[
 \PP\big(\,\that_{j} \in \mathcal{T}_{j}\,\big) \ntoinfty 1  
\]
and at the same time the empirical process 
$\sqrt{n}(P_{n}-P)$ indexed by the class of functions 
\begin{equation} \label{eq: set of functions F}
 \mathcal{F} = \big\{
   (\yb,\xb) \mapsto 
   \ind\{\tilde{t}_{1}(y_1; \xb)  \leq z_1,\dotsc,\tilde{t}_{d}(y_d; \xb)  \leq z_{d} \}\,;\, 
    z_1,\dotsc,z_d \in \RR,\, \tilde{t}_1 \in \mathcal{T}_{1},\dotsc, \tilde{t}_d \in \mathcal{T}_{d}
 \big\} 
\end{equation}
is asymptotically uniformly equicontinuous in probability 
with respect to the semimetric 
\begin{equation} \label{eq: semimetric rho}
 \rho(f_1,f_2) = P|f_{1}-f_{2}| = \E\, |f_{1}(\Yb,\Xb) - f_{2}(\Yb,\Xb)|. 
\end{equation}
\end{assumpT}

\begin{assumpT} \label{assump: true tj}
 For each $j \in \{1,\dotsc,d\}$ there exists a function 
 $t_{j} \in \mathcal{T}_{j}$ such that for each $y \in S_{Y_{j}}$ 
 and  $\xb \in S_{\Xb}$
\[
 \that_{j}(y;\xb) \inPr t_{j}(y;\xb). 
\]
Further the random vector
\[
  \epsb=(\eps_1,\dotsc,\eps_d)\tr, \ \text{where } \eps_{j}=t_{j}(Y_{j};\Xb), 
  \ j \in \{1,\dotsc,d\}, 
\]
has a continuous cdf and is independent of~$\Xb$. 
\end{assumpT}

\begin{remark}\label{remark: Donsker}
 Note that in case of iid random vectors assumption~\ref{assump: donsker}  
 is usually justified by showing that for each $j \in \{1,\dotsc,d\}$  the class of functions
 \[
  \mathcal{F}_{j} = \Big\{
   (y,\xb) \mapsto 
   \ind\{\tilde{t}(y; \xb)  \leq z \}\,  ; \, 
    z \in \RR, \tilde{t} \in \mathcal{T}_{j}
 \Big\}
 \]
 is $P_{j}$-Donsker, where $P_{j}$ is the measure of the random vector 
 $\binom{Y_j}{\Xb}$.  Validity of assumption~\ref{assump: donsker} then follows by 
 Example~2.10.8 of \cite{vaart_wellner}.   
 Nevertheless the formulation of \ref{assump: donsker}  
 allows also for dependent random variables. See  for instance \cite{noh_copula_nts_2019} 
 where \ref{assump: donsker} is verified in the context multivariate 
 nonparametric AR-ARCH times series that satisfies an appropriate $\beta$-mixing 
 assumption. Another application can be found in 
 Section~\ref{subsec: linear model with beta mix}. 
\end{remark}

\noindent \textbf{Assumptions on the quality of 
\texorpdfstring{$\widehat{t}_{j}$}{hat tj}.}
Roughly speaking assumptions~\ref{assump: donsker} 
and \ref{assump: true tj} justify that in representation~\eqref{eq: reprentation of Geps} 
the process $\AAA_{n}$ has really the form~\eqref{eq: An typical}. 
Further we specify the assumptions so that %
the process $\mathbb{B}_{n}$ is of the form~\eqref{eq: Bn typical}  
with appropriately defined processes~$\ZZZ_{jn}$. 

For a fixed~$\xb \in S_{\Xb}$ 
let $\that_{j}^{-1}(\cdot; \xb)$ be a (possibly generalized) 
inverse function to $\that_{j}(\cdot; \xb)$, i.e. 
\[
 \that_{j}^{-1}(z; \xb) = \inf\big\{ y \in \RR: \that_{j}(y; \xb) \geq z \big\}. 
\]
Now for $j \in \{1,\dotsc,d\}$  let $F_{j\eps}$ be the cdf of~$\eps_{j}$ and denote 
\begin{equation}
\label{eq: Zjn}
Z_{jn}(u; \xb) = F_{j\eps}\Big(
 t_{j}\big\{\that_{j}^{-1}\big(\widetilde{F}_{nj\eps}^{-1}(u); \xb\big); \xb \big\}
 \Big) 
 - u  
\quad \text{ and } \quad  
 \ZZZ_{jn}(u) = \sqrt{n}\,\E_{\Xb} Z_{jn}(u; \Xb),  
\end{equation}
where $\{\widetilde{F}_{nj\eps}\}$ is a sequence of strictly increasing cdfs that satisfy 
\eqref{eq: tilde Fnjeps close to Fjeps}. Note that the existence of such a sequence is guaranteed 
by the continuity of~$F_{j\eps}$ which is assumed in \ref{assump: true tj}.

Now we are ready to formulate assumptions on $Z_{jn}(u; \xb)$ so that 
one can justify that the process~$\mathbb{B}_{n}$ 
in representation~\eqref{eq: approx of Cn with HD} is really 
of the form~\eqref{eq: Bn typical} and at the 
same time the properties of processes~$\ZZZ_{jn}(u)$ match with the 
corresponding Hadamard differentiability result. 

We will introduce two versions (namely \ref{assump: process Z strict} and \ref{assump: process Z less strict})  of the assumptions on $Z_{jn}(u; \xb)$. While \ref{assump: process Z strict} is more strict and matches with the less strict assumption on the copula function \ref{assump: copula1}, assumption \ref{assump: process Z less strict} is milder but requires 
the more strict assumption on the copula \ref{assump: copula2}.

\begin{assumpZ} \label{assump: process Z strict}
For each $\epsilon > 0$ there exist a function~$M$ defined on~$S_{\Xb}$ 
and a bounded function $r$ defined on $(0,1)$ that satisfies 
\eqref{eq: r is diminishing at borders} such that for each $j \in \{1,\dotsc,d\}$, 
each $u \in \big[\frac{\epsilon}{\sqrt{n}}, 1 -  \frac{\epsilon}{\sqrt{n}}\big]$  
\begin{equation*} %
  |Z_{jn}(u; \Xb)| 
   \leq M(\Xb)\,\big[r(u)\,O_{P}\big(\tfrac{1}{\sqrt{n}}\big) 
   + o_{P}\big(\tfrac{1}{\sqrt{n}}\big)\big] 
  \quad \text{with} \quad
 \E\, M(\Xb) < \infty, 
\end{equation*}
where the terms $O_{P}\big(\tfrac{1}{\sqrt{n}}\big)$, $o_{P}\big(\tfrac{1}{\sqrt{n}}\big)$ depend neither on~$u$ nor $\Xb$. 
\end{assumpZ}

In what follows $\beta$ will be the constant from assumption~\ref{assump: copula2}.

\begin{assumpZ} \label{assump: process Z less strict}
For each $\epsilon >0$ there exists a function~$M$ defined on~$S_{\Xb}$ 
such that:
\begin{itemize}
 \item for $\beta = 0$ one has 
\begin{equation*} 
 \sup_{u \in \big[ \frac{\epsilon}{n^{1/2}}, 1 -  \frac{\epsilon}{n^{1/2}}\big]}  
  |Z_{jn}(u; \Xb)|  = M(\Xb)\,o_{P}\big(n^{-1/4}\big)
  \text{ with }\ \E M^{2}(\Xb) < \infty;  
\end{equation*}
\item for $\beta > 0$ there exist constants $\alpha \geq 0$, $\gamma \in [0,\frac{1}{4}]$ and $s \geq 2$ so that 
\begin{equation*} %
 \sup_{u \in \big[ \frac{\epsilon}{n^{\vartheta/2}}, 1 -  \frac{\epsilon}{n^{\vartheta/2}}\big]}  
  \frac{|Z_{jn}(u; \Xb)|}{u^{\alpha}(1-u)^{\alpha}} 
  = M(\Xb)\,o_{P}\big(n^{-(1/4+\gamma)}\big),  
\end{equation*}
where 
\begin{equation} \label{eq: gamma}
  \gamma \geq 
  \frac{\frac{s-2}{s-1}(\beta - \alpha)_{+}}{4\big(1 - \alpha - (\beta-\alpha)\frac{s}{s-1}\big)_{+}}, 
  \quad 
  \Es [M(\Xb)]^{s} < \infty,   
\end{equation}
and 
\begin{equation} \label{eq: vartheta}
  \vartheta = \min\bigg\{\frac{\frac{s-2}{s-1}+4 \gamma \frac{s}{s-1}}{2(1-\alpha)}, 1 \bigg\}.   
\end{equation}
\end{itemize}
Moreover the term $o_{P}(\cdot)$ in either of the 
versions of the assumption depends neither on~$u$ nor~$\Xb$.
\end{assumpZ}

\begin{remark} \label{rem: Zn in view of assumptions} 
Regarding \ref{assump: process Z less strict} note that typically one is interested 
in situations when $\vartheta = 1$ as this implies the weak convergence of the copula 
process on $[0,1]^{d}$. To achieve that one needs 
\begin{equation*} %
 s \geq \frac{(4-2\alpha)_{+}}{\big(1+4 \gamma - 2(1-\alpha)\big)_{+}}\,. 
\end{equation*}
\end{remark}

\subsection{General results on empirical copulas based on pseudo-observations}

\begin{theorem} \label{thm for param adjustment}
Suppose that  assumptions \ref{assump: donsker}, 
\ref{assump: true tj}, \ref{assump: copula1}, 
and \ref{assump: process Z strict}  
are satisfied. Then  $\eqref{eq: equiv of Cn and oracle Cn}$ holds, i.e.\
$$ \sup_{\ub \in [0,1]^d} \sqrt{n}\,\big|\Chat_{n}(\ub) 
 -  C_{n}^{(or)}(\ub) \big| = o_{P}(1).    
$$
\end{theorem}

\begin{theorem} \label{thm for nonparam adjustment}
Suppose that \ref{assump: donsker}, 
\ref{assump: true tj}, \ref{assump: copula2}, 
and \ref{assump: process Z less strict}  
are satisfied. Then for each $\epsilon > 0$
\begin{equation*} %
 \sup_{\ub \in \big[ \frac{\epsilon}{n^{\vartheta/2}}, 1 -  \frac{\epsilon}{n^{\vartheta/2}}\big]^{d}} \sqrt{n}\,\big|\Chat_{n}(\ub) 
 -  C_{n}^{(or)}(\ub) \big| = o_{P}(1).    
\end{equation*}
\end{theorem}

Note that although Theorem~\ref{thm for nonparam adjustment} does not 
guarantee the asymptotic equivalence of $\Chat_{n}$ and $C_{n}^{(or)}$ 
on the whole $d$-dimensional unit cube $[0,1]^{d}$, the important thing is that 
$\big[ \frac{\epsilon}{n^{\vartheta/2}}, 1 -  \frac{\epsilon}{n^{\vartheta/2}}\big]^{d}$ 
is expanding to $[0,1]^{d}$. This is often enough to prove 
for instance the asymptotic equivalence of moment-like estimators based 
on estimated pseudo-observations ($\epshat_{ji}$) and estimators based on the unobserved   $\eps_{ji}$. See Section~4.3 of \cite{cote_genest_omelka_2019} where it is shown that  
even a weaker result is sufficient for some type of inference. 

If $\vartheta = 1$, that is if $\gamma$ and $\alpha$ are `sufficiently large' 
then one gets the stronger statement of Theorem~\ref{thm for param adjustment}. 
This is formulated in the following corollary.

\begin{corollary} \label{cor for nonparam adjustment}
 Suppose that the assumptions of Theorem~\ref{thm for nonparam adjustment} are satisfied. If either $\beta = 0$ or 
$4\gamma + 2 \alpha \,\frac{s}{s-1} \geq 1$ then $\eqref{eq: equiv of Cn and oracle Cn}$ holds. 
\end{corollary}

Note that one can view \ref{assump: process Z less strict} as a more 
general version of assumption (\textbf{Yn}) from \cite{ogv_SCAJS_2015}. 
Note that (\textbf{Yn}) does not only require that $\alpha=\beta=\frac{1}{2}$ 
but more importantly it requires a bounded covariate (i.e. $s=\infty$).

\section{Some specific examples} \label{sec: applications}
In this section we illustrate how the results presented in this paper 
can be used in the specific models. For simplicity of presentation we 
start with the linear model with iid errors (Section~\ref{subsec: lin model}). 
This model will be then generalized to $\beta$-mixing errors 
(Section~\ref{subsec: linear model with beta mix}) or functional 
linear model (Section~\ref{subsec: functional lin model}). The section 
is concluded by the application to location-shape-scale models (Section~\ref{subsec: location scale shape models}).

\subsection{Linear model with iid errors} \label{subsec: lin model}
Consider the standard linear model for each of the marginals (with the same $k$-dimensional covariates $\Xb_1,\dotsc,\Xb_n$), i.e. 
\begin{equation} \label{eq: linear model}
 Y_{ji} = \Xb_{i} \tr \bb_{j} + \eps_{ji}, \quad i=1,\dotsc,n, \quad j\in\{1,\dotsc,d\}, 
\end{equation}
where the vector of centred errors $\epsb_{i}=(\eps_{1i},\dotsc,\eps_{di})\tr$ is independent of $\Xb_{i}$. 
Thus one can take simply $t_j(y;\xb)=y - \xb \tr \bb_{j}$ and $\widehat t_j(y;\xb)= y - \xb \tr \widehat{\bb}_{j}$, 
where $\widehat{\bb}_{j}$ is for instance least squares estimator of $\bb_{j}$. 
Then under mild assumptions (including among others that $\E\, \|\Xb_{1}\|^2 < \infty$ with 
 $\|\cdot\|$ being the Euclidean norm) $\widehat{\bb}_{j}$ is $\sqrt{n}$-consistent. Further, for some $\delta > 0$, 
 one can define
\[
\mathcal{F}_j=\big\{(y,\xb)\mapsto \ind\{y - \xb\tr\widetilde{\bb}  \leq z\}\,;\, z\in\RR,\; 
\widetilde{\bb} \in \RR^{k} ,\,
   \|\widetilde{\bb}-\bb_{j}\|\leq \delta \big\} ,  
\]
which is $P_j$-Donsker ($j \in \{1, \dots,d\}$) and thus  \ref{assump: donsker} is satisfied by Remark \ref{remark: Donsker}.
Further note that 
\[
Z_{jn}(u; \Xb) = F_{j\varepsilon} \big(\widetilde{F}_{nj\varepsilon}^{-1}(u)+
 \Xb\tr(\widehat{\bb}_{j} - \bb_{j})\big) - u, 
 \quad u \in [0,1]. 
\]
So with the help of the mean value theorem
\[
 Z_{jn}(u; \Xb) =  F_{j\varepsilon} \big(\widetilde{F}_{nj\varepsilon}^{-1}(u)\big) - u
  + f_{j\eps}(\xi_{n,u}^{\Xb}) \Xb\tr(\widehat{\bb}_{j} - \bb_{j}), 
\]
with $\xi_{n,u}^{\Xb}$ between $\widetilde{F}_{nj\varepsilon}^{-1}(u)$ and 
$\widetilde{F}_{nj\varepsilon}^{-1}(u) + \Xb\tr(\widehat{\bb}_{j} - \bb_{j})$. 
Thus with the help of~\eqref{eq: tilde Fnjeps close to Fjeps} 
it is straightforward to show that \ref{assump: process Z strict} 
holds provided that there exists (a version of the) density~$f_{j\eps}$ which is bounded 
and satisfies
\begin{equation*} %
 \lim_{u \to 0_{+}} f_{j\eps}\big(F_{j\eps}^{-1}(u)\big) 
 = 0 = \lim_{u \to 1_{-}} f_{j\eps}\big(F_{j\eps}^{-1}(u)\big).  
\end{equation*}

On the other hand the straightforward application of the Hadamard differentiability result of 
\cite{bucher2013empirical} would require (among others) the weak convergence of the 
process 
\[
\ZZZ_{jn}(u) = \sqrt{n}\,\E_{\Xb}\big[F_{j\varepsilon} \big(\widetilde{F}_{nj\varepsilon}^{-1}(u)+
 \Xb\tr(\widehat{\bb}_{j} - \bb_{j})\big)  -u \big], 
 \quad u \in [0,1], 
\]
which is a more delicate task. Moreover to show this weak convergence 
it seems to be necessary to add the assumption of \textbf{the continuity} of 
$f_{j\eps}\big(F_{j\eps}^{-1}(u)\big)$ which is not required by our approach.

\subsection{Linear regression with \texorpdfstring{$\beta$}{beta}-mixing observations} \label{subsec: linear model with beta mix}

Note that Theorem \ref{thm for param adjustment} gives conditions to obtain asymptotic equivalence of the residual-based empirical copula process and the one based on true errors even in models with dependent observations. Assume observations $\binom{\Yb_{1}}{\Xb_{1}},\dotsc,\binom{\Yb_{n}}{\Xb_{n}}$ from a strictly stationary $\beta$-mixing sequence $\binom{\Yb_{i}}{\Xb_{i}}$, $i\in\mathbb{Z}$,  fulfilling the linear model (\ref{eq: linear model}), where the errors $\epsb_{i}=(\eps_{1i},\dots,\eps_{di})\tr$ are independent of past and present covariates $\Xb_\ell$, $\ell\leq i$. For the $\beta$-mixing coefficients we assume $\beta_i=O(i^{-b})$ for some $b>1$.

Provided that consistent estimators $\widehat{\bb}_1,\dotsc,\widehat{\bb}_d$ are available  
it is sufficient to  verify condition \ref{assump: donsker} for the function class
\begin{multline*} 
 \mathcal{F}=\Big\{(\yb,\xb)\mapsto \ind\{y_1 - \xb\tr\widetilde{\bb}_1  \leq z_1,\dotsc, y_d - \xb\tr\widetilde{\bb}_d  \leq z_d\}\,;\, z_1,\dotsc,z_d\in\RR,\; 
 \\
 \forall_{j\in\{1,\dotsc,d\}}\  \widetilde{\bb}_j \in \RR^{k}\, \&\,
   \|\widetilde{\bb}_{j}-\bb_{j}\|\leq \delta \Big\}  
\end{multline*}
for some $\delta > 0$. 
To do that we will follow the approach of  \cite{dedecker_louhichi} 
and consider the seminorm
\begin{equation*}  
 \|f\|_{2,\beta}^{2} = \int_{0}^{1} \beta^{-1}(u)\,Q_{f}^{2}(u)\,du, 
\end{equation*}
where 
$ \beta^{-1}(u) = \inf\{x>0 :\beta_{\lfloor x\rfloor}\leq u\}$ and $  
 Q_{f}(u) = \inf\big\{x>0: \PP\big(|f(\Yb_{1}, \Xb_{1})| > x \big) \leq u \big\}$.
 In \cite{noh_copula_nts_2019} (see their formula (A.13) in section A2) 
 it was derived that there exists a finite constant $K$ such that 
\begin{equation}\label{eq: semi-norm-beta}
\| f-g\|_{2,\beta}^2\leq \frac{K\,b}{b-1} (P|f-g|)^{(b-1)/b}
\end{equation}
for all indicator functions $f$ and $g$. Denote $\|\cdot\|_2$ the $L^2(P)$-norm.  
Then similarly as in the proof of Lemma~1 of \cite{dette2009goodness} one can show 
that the bracketing integral condition
\begin{equation}\label{eq: bracketing-beta}
\int_{0}^\infty \sqrt{\log N_{[\,]}(\epsilon, \mathcal{F},\|\cdot\|_{2,\beta}})\, d\epsilon<\infty
\end{equation}
is fulfilled when for each $j \in \{1,\dotsc,d\}$ 
\[
 \int_{0}^1 \sqrt{\log N_{[\,]}(\epsilon^{2b/(b-1)}, \mathcal{M}_j,\|\cdot\|_{2}})\, d\epsilon<\infty,  
\]
where 
\[
\mathcal{M}_j=\big\{\xb\mapsto \xb\tr (\widetilde{\bb}-\bb_{j})\, ;\, 
\widetilde{\bb} \in \RR^{k}\, \&\,  \|\widetilde{\bb} -\bb_{j}\|\leq \delta \big\}. 
\]
But the bracketing number $N_{[\,]}\big(\epsilon^{2b/(b-1)}, \mathcal{M}_j,\|\cdot\|_{2}\big)$ 
is of order $\epsilon^{-2bd/(b-1)}$ by applying Theorem 2.7.11 in \cite{vaart-wellner-2007}. 
Thus the bracketing 
integral~\eqref{eq: bracketing-beta} is finite which implies asymptotic equicontinuity of the empirical process indexed in $\mathcal{F}$ with respect to the semi-norm $\|\cdot\|_{2,\beta}$ 
\citep[see Section~4.3 of][]{dedecker_louhichi}.  Now using once more the inequality in 
\eqref{eq: semi-norm-beta} yields that \ref{assump: donsker} holds. Further \ref{assump: true tj} follows from consistency of the estimator for the regression function. 

Now condition \ref{assump: process Z strict} can be verified similarly as in Section~\ref{subsec: lin model} provided that the estimator $\widehat{\bb}_j$  is $\sqrt{n}$-consistent. For example for the least squares estimator this is a simple consequence of the law of large numbers and central limit theorem for $\beta$-mixing sequencies and is fullfilled under   existence of $m>2$ moments of covariates and errors if $b>m/(m-2)$ \citep[see e.g.\ Proposition 2.8 and Theorem~2.21 in][]{fan_yao_book2005}.

\begin{remark}
For simplicity here (as well in Section~\ref{subsec: lin model}) we consider 
only linear models so far. But analogously one can consider nonlinear or even non- or semiparametric regression models provided that suitable regression estimators are available. See for instance \cite{ogv_SCAJS_2015} and \cite{noh_copula_nts_2019} where nonparametric location-scale models were 
considered. 
\end{remark}

\subsection{Functional linear model} \label{subsec: functional lin model}

Now assume that  $\binom{\Yb_{1}}{X_{1}},\dotsc,\binom{\Yb_{n}}{X_{n}}$  is a random sample 
from the generic distribution $\binom{\Yb}{X}$, where $X$ is a functional covariate such that 
$$
 Y_{j}= \langle X,b_{j}\rangle + \varepsilon_{j},\quad j\in\{1,\dotsc,d\},  
$$
and the error vector $\epsb=(\eps_{1},\dotsc,\eps_{d})\tr$ is independent of~$X$.

Suppose that the covariate $X$ as well as the true parameter function 
$b_{j}$ are random elements of the Hilbert space
$L^2([0,1])$ with the inner product $\langle f,g\rangle=\int_0^1 f(t)g(t)\,dt$ and norm $\|f\|_2=\langle f,f\rangle^{1/2}$. For simplicity assume $X\geq 0$.

Assume that for each $j$  there is an estimator $\widehat{b}_j$ based on an iid sample $\binom{Y_{ji}}{X_{i}}$,  %
$i=1,\dotsc,n$, such that
\begin{equation}\label{rate beta}
\| \widehat{b}_j-b_{j}\|_2=o_P(n^{-1/4+\gamma}) 
\end{equation}
for some $\gamma\geq 0$ corresponding to $\gamma$ in assumption \ref{assump: process Z less strict}.
Convergence rates for estimators in the functional linear model can be found in \cite{hall2007methodology}, \cite{yuan2010reproducing} 
or \cite{shang2015nonparametric}, 
among others. For example, under the assumption that $b_{j}$ is an element of the univariate Sobolev-space $\mathcal{W}_2^m([0,1])$ for some $m\geq 1$, condition (\ref{rate beta}) 
is fulfilled for the regularized estimators in  \cite{yuan2010reproducing} under the assumptions of their Corollary 11.

We further assume that $\PP\big(\,\widehat{b}_j-b_{j}\in\mathcal{G}\big)\to 1$ as $n\to\infty$ for a function class $\mathcal{G}$ such that the bracketing number fulfills
\begin{equation}\label{eq: bracketing of G}
 \log N_{[\,]}(\mathcal{G},\epsilon,\|\cdot\|_2)\leq \frac{K}{\epsilon^{1/k}} 
\end{equation}
for some $K>0$ and $k>1$. For example for $k=2$ this is satisfied with the Sobolev unit ball given by  
\[
\mathcal{G}=\big\{b\in \mathcal{W}_2^{2}([0,1]): \|b\|_2+\|b^{(2)}\|_2\leq 1\big\}, 
\] 
where $b^{(2)}$ stands for the second derivative of~$b$. 
Note that $\|\cdot\|_2$-bracketing numbers can be bounded by $\|\cdot\|_\infty$-covering numbers, such that Corollary 4.3.38 
in \cite{gine2021mathematical}
can be applied. Similar results can be found in Example 19.10 in 
\cite{vaart2000asymptotic} 
or Corollary 4 in \cite{nickl2007bracketing}. 
We obtain $\PP\big(\,\widehat{b}_j-b_{j}\in\mathcal{G}\big)\to 1$ for $\mathcal{G}$ as above for example under the assumption $b_{j}\in\mathcal{W}_2^{m}([0,1])$ for some $m>2$ for the estimator in \cite{yuan2010reproducing} also chosen from $\mathcal{W}_2^{m}([0,1])$ under the assumptions of their Corollary 11, because then $\|\widehat{b}_j-b_{j}\|_2+\|\widehat{b}^{(2)}_j-b_{j}^{(2)}\|_2\to 0$. Other estimators and subspaces $\mathcal{G}$ of $L^2([0,1])$ could be used as well.

To derive conditions under which assumption \ref{assump: donsker} is valid introduce 
\[
 t_j(Y_j;X) = Y_j - \langle X,b_{j}\rangle, 
\qquad 
\widehat{\eps}_j = \widehat{t}_j(Y_j;X) = Y_j - \langle X,\widehat{b}_j\rangle 
\]
and note that 
\[
\mathcal{F}_j =\big\{(y,x)\mapsto \ind\{y - \langle x,b\rangle \leq z \}  
\, ;\, z\in\mathbb{R}, b-b_{j}\in\mathcal{G} \big\}. 
\]
To show that this function class is Donsker we derive an upper bound for the bracketing number. To this end let $\epsilon>0$ and let $[b_i^L,b_i^U]$, $i=1,\dotsc,N(\epsilon)=O(\exp(\epsilon^{-2/k}))$ be brackets for $\mathcal{G}$ of $\| \cdot\|_2$-length $\epsilon^2$ (see assumption (\ref{eq: bracketing of G})). 
Note that for $x\geq 0$ and for $b - b_{j}$ from the bracket 
$[b_i^L,b_i^U]$ we obtain that the indicator function $\ind\{y - \langle x,b\rangle \leq z \}$ is contained in the bracket 
\[
\big[\,\ind\{y - \langle x,b_{j} + b_i^L\rangle \leq z \},\, 
\ind\{y - \langle x,b_{j} + b_i^U\rangle \leq z \}
\big]
\]
for each $z \in\mathbb{R}$. Further the above  
bracket has $L^2(P_{j})$-length (where $P_{j}$ denotes the distribution of 
$\binom{Y_j}{X}$) bounded by
\begin{align*}
\Big(\Es & \big[\ind\{Y_j - \langle X,b_{j}\rangle \leq z  +\langle X,b_i^U\rangle\}
- \ind\{Y_j - \langle X,b_{j}\rangle \leq z + \langle X,b_i^L\rangle\}\big]^2\Big)^{1/2}\\
&\leq \left( \Es\big[F_{j\varepsilon}(z+\langle X,b_i^U\rangle)-F_{j\varepsilon}(z+\langle X,b_i^L\rangle)\big] \right)^{1/2}\\
&\leq \left(\|f_{j\varepsilon}\|_\infty \Es\big[\langle X,b_i^U-b_i^L\rangle\big]\right)^{1/2}
\leq \left(\|f_{j\varepsilon}\|_\infty \Es\| X\|_{2}\,\epsilon^2\right)^{1/2} =O(\epsilon),
\end{align*}
where the last line follows from the Cauchy-Schwarz inequality  
provided that one assumes a bounded density $f_{j\varepsilon}$ and $\Es\| X\|_2<\infty$. 
 Similar to the proof of Lemma 1 in \cite{akritas_keilegom_2001} 
 from this one obtains an upper bound $O(\epsilon^{-2}\exp(\epsilon^{-2/k}))$ for the  $L^2(P_{j})$-bracketing number of the class $\mathcal{F}_j$. Thus $\mathcal{F}_j$ is Donsker by the bracketing integral condition in Theorem~19.5 of \cite{vaart2000asymptotic} for $k > 1$ and  \ref{assump: donsker} is fulfilled by Remark~\ref{remark: Donsker}. 

Further note that by the mean value theorem and \eqref{eq: tilde Fnjeps close to Fjeps}
\[
Z_{jn}(u;X) = F_{j\varepsilon} \big(\widetilde{F}_{nj\varepsilon}^{-1}(u)+\langle X,\widehat{b}_j-b_{j}\rangle\big)-u
= f_{j\varepsilon} (\xi_{n,u}^{X})\langle X,\widehat{b}_j-b_{j}\rangle 
 + o\big(\tfrac{1}{\sqrt{n}}\big)
\]
with $\xi_{n,u}^{X}$ converging to $F_{j\varepsilon}^{-1}(u)$. Now using \eqref{rate beta} one can bound 
\[
 |Z_{jn}(u;X)|\leq f_{j\varepsilon} (\xi_{n,u}^{X})\,\|X\|_2 \, o_P(n^{-1/4+\gamma}) 
 + o\big(\tfrac{1}{\sqrt{n}}\big), 
\]
where the $o_{P}$ as well as $o$ term do not depend on $u$ and $X$.  As this term is typically larger than $O_{P}(n^{-1/2})$ one can use Theorem \ref{thm for nonparam adjustment} that requires \ref{assump: process Z less strict} (instead of more strict \ref{assump: process Z strict}).  With some further effort it possible to find appropriate moment assumptions for $\|X\|_2$ and smoothness assumption on $f_{j\varepsilon}\big(F_{j\varepsilon}^{-1}(u)\big)$ so that  \ref{assump: process Z less strict} holds.

It is worth noting that trying to show the weak convergence of the process 
\[
 \ZZZ_{jn}(u) = \sqrt{n}\,\E_{X}\big[F_{j\varepsilon} \big(\widetilde{F}_{nj\varepsilon}^{-1}(u)
 +\langle X,\widehat{b}_j-b_{j}\rangle\big)-u\big]
= \sqrt{n}\,\E_{X}\big[f_{j\varepsilon} (\xi_{n,u}^{X})\langle X,\widehat{b}_j-b_{j}\rangle \big]  
+ o\big(\tfrac{1}{\sqrt{n}}\big)
\]
could be rather tough here as the rate of convergence 
of terms like $\langle X,\widehat{b}_j-b_{j}\rangle$ is typically slower than $n^{-1/2}$ \citep[see e.g.][]{cardot2007clt, shang2015nonparametric, yeon_et_al_2023}. 

\subsection{Location-scale-shape models} \label{subsec: location scale shape models}

In this subsection we consider generalizations of location-scale models. Assume that $Y$ is a real-valued random variable with a cdf
\[
 \PP(Y\leq y)=\Psi\big(\tfrac{y-\alpha}{\beta} ;\gamma\big)
\]
with location parameter $\alpha$, scale parameter $\beta>0$ shape parameter 
$\gamma$ and a known function~$\Psi$. As an example one can consider 
the skew-normal distribution \citep{azzalini} where $\Psi$ 
is given by 
\[
\Psi_{s} (z;\gamma)= \int_{-\infty}^z 2\varphi(t)\Phi(\gamma t)\, dt\\
\]
Alternatively one can consider for instance epsilon-skew-normal distribution \citep{mudholkar-hutson} 
or  generalized normal distribution \citep[][]{nadarajah2005generalized}. 

Now consider our model with observations $(Y_{ji},\Xb_i)$, $i=1,\dotsc,n$, where for each $j\in\{1,\dotsc,d\}$, the conditional distribution of $Y_{ji}$, given $\Xb_i=\xb$, is of the above form with parameters depending on the covariate, i.e. 
$$
F_{j\xb}(y)=\PP\big(Y_{ji}\leq y\,|\,\Xb_i=\xb\big)=\Psi_{j}\left(\frac{y-\alpha_j(\xb)}{\beta_j(\xb)} ;\gamma_j(\xb)\right), 
$$
where $\Psi_{j}$ is know but possibly different for $j \in \{1,\dotsc,d\}$.

Set $t_j(y;\xb)=F_{j\xb}(y)$, such that each $\varepsilon_{ji}=t_j(Y_{ji};\Xb_i)$ is uniformly distributed on $[0,1]$ and thus independent of $\Xb_i$. 
Further, assume that also the random vector $\epsb_{i}=(\eps_{1i},\dotsc,\eps_{di})\tr$ is independent of $\Xb_i$. Then  the the joint cdf of 
$\epsb_{i}$ coincides with the copula function~$C$ and can 
be estimated by $\widehat C_n$ as in (\ref{eq: definition C_n}) based on pseudo-observations
$$
 \widehat\varepsilon_{ji}=\widehat t_j(Y_{ji};\Xb_i)
=\Psi_j\left(\frac{Y_{ji}-\widehat\alpha_j(\Xb_i)}{\widehat\beta_j(\Xb_i)} ;\widehat\gamma_j(\Xb_i)\right)
$$
 if consistent estimators $\widehat\alpha_j,\widehat\beta_j,\widehat\gamma_j$ for the parameter functions are available.

 In the models without covariates parameters are typically estimated with method of moments or maximum likelihood and are $\sqrt{n}$-consistent with asymptotic normal distribution under regularity assumptions. To the authors' knowlege estimators for covariate-dependent parameters have not yet been investigated in the literature. Thus to obtain asymptotic results for the copula estimator it is desirable to require only weak assumptions on the parameter function estimators. One possibility to show assumption  \ref{assump: donsker} is to find estimators and function classes 
$\mathcal{F}_j^{(\ell)}$, $\ell \in \{1,2,3\}$, such that 
$\PP\big(\widehat\alpha_j\in \mathcal{F}_j^{(1)}\big)\to 1$, 
$\PP\big(\widehat\beta_j\in \mathcal{F}_j^{(2)}\big)\to 1$, 
$\PP\big(\widehat\gamma_j\in \mathcal{F}_j^{(3)}\big)\to 1$ 
and the corresponding empirical process indexed by 
\[
\mathcal{F}_j=\Big\{(y,\xb)\mapsto \ind\big\{y\leq \alpha(\xb) + \beta(\xb)\Psi_j^{-1}\big(z;\gamma(\xb)\big)\big\}\,;\, z\in\mathbb{R}, \alpha\in \mathcal{F}_j^{(1)},\beta\in \mathcal{F}_j^{(2)},\gamma\in \mathcal{F}_j^{(3)}\Big\}
\]
is $P_j$-Donsker, see Remark \ref{remark: Donsker}. 

Further one has to consider 
\begin{eqnarray*}
Z_{jn}(u;\Xb)&=&t_j\big(\widehat t_j^{-1}(u;\Xb);\Xb\big)-u\\
&=& \Psi_j\left(\frac{\widehat\alpha_j(\Xb)-\alpha_j(\Xb)+\widehat\beta_j(\Xb)\Psi_j^{-1}(u;\widehat\gamma_j(\Xb))}{\beta_j(\Xb)};\gamma_j(\Xb)\right)-u
\end{eqnarray*} 
and it should be easier to derive conditions to show 
either \ref{assump: process Z strict} or \ref{assump: process Z less strict} using a Taylor expansion and convergence rates of the parameter estimators than to show weak convergence of the process defined in (\ref{eq: joint process A_n Z_j}) 
with $\mathbb{Z}_{jn}(u)=\sqrt{n}\,\E_{\Xb}Z_{jn}(u;\Xb)$. The latter weak convergence would be needed without the new Hadamard differentiability result. 

Note that in this example one could estimate~$C$ also by the empirical cdf 
of $(\epshat_{1i},\dotsc,\epshat_{di})\tr$,  $i=1,\dotsc,n$, instead of the empirical copula function $\widehat{C}_{n}$. However,  the weak convergence of the process 
$\mathbb{B}_n$ would be needed to derive the limit distribution of such estimator.  
But this  can be avoided by considering the empirical copula estimator~\eqref{eq: definition C_n} and the results presented in this paper. 

\section{Conclusions and further discussions}
\label{sec: discussion}

In our paper we presented two generalizations of the Hadamard-differentiability 
of \cite{bucher2013empirical} that are motivated by dealing with the empirical copulas 
in the presence of covariates. It is worth noting that these results can be used 
in many other situations provided that appropriate results on the estimates are available. 
One can (among others) think of for instance long-range dependence where one would 
get different than $\sqrt{n}$-rates of convergence. A different 
(but rather straightforward)  
generalization would be to consider appropriately weighted differentiability result 
to get the weighted empirical copula approximation as in \cite{cote_genest_omelka_2019}.

\appendix 
\section{Proofs of the results in Section~\ref{sec: HD}} \label{appendix A}

\subsection{Proof of Theorem~\ref{thm HD of copula fctional}}

Let $C_{n}^{h,\htilde} \in \mathcal{D}_{n}^{(1)}$. 
Note that thanks to~\eqref{eq: r is diminishing at borders} 
(which implies that $\htilde_{k}(1)=0$) and~\eqref{eq: tilde hn} 
\[
 \widetilde{h}(\ub^{(j)}) 
 =  \htilde_{j}(u_{j}) + \sum_{k=1, k\neq j}^d C^{(k)}\big(\ub^{(j)}\big)\,\htilde_{k}(1)
 =  \htilde_{j}(u_{j}), 
\]
where $\ub^{(j)}$ was introduced in~\eqref{eq: ubj}. 
Further recall that $h_{j}(u_{j}) = h(\ub^{(j)})$.

Now let $U$ stand for the cdf 
of a random variable with the uniform distribution on~$[0,1]$. Then 
the marginal cdf  for the $j$-th coordinate 
of the joint cdf 
$C_{n}^{h,\htilde} = C + t_{n}\,h + t_{n}\,\widetilde{h}$ is given by 
\[
 \big(C + t_{n}\,h + t_{n}\,\widetilde{h}\big)(\ub^{(j)}) 
 = \big(U + t_{n}\,h_{j} + t_{n}\,\htilde_{j}\big)(u_{j}). 
\]
Denote its corresponding generalized inverse function as 
\begin{equation*} %
 \xi_{jn}^{h_{j},\htilde_{j}}(u) = \big(U + t_{n}\,h_{j} + t_{n}\,\htilde_{j}\big)^{-1}(u).  
\end{equation*}
For simplicity of notation introduce
\begin{equation*} %
 \xib_{n}^{h,\widetilde{h}}(\ub) = \big(\xi_{1n}^{h_{1},\widetilde{h}_{1}}(u_{1}), \dotsc,\xi_{dn}^{h_{d},\widetilde{h}_{d}}(u_{d})\big)\tr. 
\end{equation*}

Note that  applying Lemma~\ref{lemma HD of general inverse second}(i) 
(with $t_{n}=\ttilde_{n}$) given in Section~\ref{subsec HD of quantile} 
for each $j \in \{1,\dotsc,d\}$ implies that 
$\xib_{n}^{h,\widetilde{h}}(\ub) \ntoinfty \ub$ (uniformly in 
$C_{n}^{h,\htilde} \in \mathcal{D}_{n}^{(1)}$ 
and $\ub \in [0,1]^{d}$).

Now using the asymptotic uniform equicontinuity 
of $h \in \mathcal{A}_{n}$
and the form of~$\widetilde{h}$ given in~\eqref{eq: tilde hn}  one derives  
that %
 \begin{align}
\notag
  \Phi(C + t_{n}\,h + t_{n}\,\widetilde{h})(\ub) 
 &= \big(C + t_{n}\,h + t_{n}\,\widetilde{h}\big)
\big(\xib_{n}^{h,\widetilde{h}}(\ub)\big)
\\
\notag
 &= C\big(\xib_{n}^{h,\widetilde{h}}(\ub)\big) 
  + t_{n}\,h\big(\xib_{n}^{h,\widetilde{h}}(\ub)\big)
  + t_{n}\,\widetilde{h}\big(\xib_{n}^{h,\widetilde{h}}(\ub)\big) 
\\
\label{eq: expanding copula functional}
 &= C\big(\xib_{n}^{h,\widetilde{h}}(\ub)\big) 
  + t_{n}\,h(\ub)  
  + t_{n}\, \sum_{j=1}^{d} C^{(j)}(\ub)\,\htilde_{j}\big(\xi_{jn}^{h_{j},\htilde_{j}}(u_j)\big)
   + o(t_{n}),   
 \end{align}
where here (as well as in the sequel) by $o(t_{n})$ we understand a remainder term that 
may depend on~$\ub$, $h$ 
and $\widetilde{h}$ but it is uniformly asymptotically negligible, i.e.  
\[
 \sup_{C_{n}^{h,\htilde} \in \mathcal{D}_{n}^{(1)}}  
 \sup_{\ub \in [0,1]^{d}} \frac{|o(t_{n})|}{t_{n}} \ntoinfty 0. 
\]

Now we compare the quantity~$C\big(\xib_{n}^{h,\widetilde{h}}(\ub)\big)$ 
with $C(\ub)$. 
Note that by the mean value theorem  
\begin{align}
 \notag 
 \Big|C\big(\xib_{n}^{h,\widetilde{h}}(\ub)\big)
&- C(\ub) 
 - \sum_{j=1}^{d} C^{(j)}(\ub) \big[\xi_{jn}^{h_{j},\htilde_{j}}(u_{j}) - u_{j} \big]  
 \Big|  
\\ 
\label{eq: C at xi minus C at u}
 & \leq  
\sum_{j=1}^{d} \big|C^{(j)}(\ub_{n}^{h,\widetilde{h}}) - C^{(j)}(\ub) \big|\,  \big|\xi_{jn}^{h_{j},\htilde_{j}}(u_{j}) - u_{j} \big|,   
\end{align}
where $\ub_{n}^{h,\widetilde{h}}$ lies between the points 
$\xib_{n}^{h,\widetilde{h}}(\ub)$ 
and $\ub$. 

Now fix $j \in \{1,\dotsc,d\}$ and let $\delta \in (0,\tfrac{1}{4})$. Introduce 
the sets  
\begin{equation*} %
 I_{j}(\delta) = \big\{ \ub \in [0,1]^{d} : \delta \leq u_{j} \leq 1 - \delta \big\}, 
\quad  I_{j}^{c}(\delta) = [0,1]^d\setminus I_{j}(\delta). 
\end{equation*}
Then with the help of the (uniform) continuity of~$C^{(j)}$ on~$I_{j}(\delta/2)$ and 
Lemma~\ref{lemma HD of general inverse second}(ii) (with $t_n=\tilde t_n$) one gets that 
\begin{align} 
\notag 
 \sup_{C_{n}^{h,\htilde} \in \mathcal{D}_{n}^{(1)}}\, 
& \sup_{\ub \in I_{j}(\delta)}
\big\{\big|C^{(j)}(\ub_{n}^{h,\widetilde{h}}) - C^{(j)}(\ub) \big|\,  \big|\xi_{jn}^{h_j,\widetilde{h}_j}(u_{j}) - u_{j} \big| \big\}
\\
\label{eq: ineq for Cj minus Cj for uj far from edge}
\leq & \sup_{C_{n}^{h,\htilde} \in \mathcal{D}_{n}^{(1)}}\, 
\sup_{\ub \in I_{j}(\delta/2)} \big|C^{(j)}(\ub_{n}^{h,\widetilde{h}}) - C^{(j)}(\ub) \big|
\,O(t_{n})
= o(1)\, O(t_{n}) = o(t_{n}), 
\end{align}
where we have used that  by Lemma~\ref{lemma HD of general inverse second}(i) 
for all sufficiently large~$n$ one has 
\[
 \sup_{C_{n}^{h,\htilde} \in \mathcal{D}_{n}^{(1)}}\,  \xi_{jn}^{h_{j},\htilde_{j}}(1-\delta) \leq 1-\delta/2 
 \quad \& \quad 
  \inf_{C_{n}^{h,\htilde} \in \mathcal{D}_{n}^{(1)}}\,  \xi_{jn}^{h_{j},\htilde_{j}}(\delta) \geq \delta/2  
\]
and thus $\ub_{n}^{h,\widetilde{h}} \in I_{j}(\delta/2)$. 

Further as from the properties of the copula function we know that $C^{(j)} \in [0,1]$ one 
can with the help of Lemma~\ref{lemma HD of general inverse second}(ii) (for all sufficiently large~$n$) bound 
\begin{align}
 \notag 
 \sup_{C_{n}^{h,\htilde} \in \mathcal{D}_{n}^{(1)}}\, 
&\sup_{\ub \in I_{j}^{c}(\delta)}
\big\{\big|C^{(j)}(\ub_{n}^{h,\widetilde{h}}) - C^{(j)}(\ub) \big|\,  \big|\xi_{jn}^{h_{j},\htilde_{j}}(u_{j}) - u_{j} \big| \big\}
\\
\notag 
& 
\leq 2\,\sup_{C_{n}^{h,\htilde} \in \mathcal{D}_{n}^{(1)}}\,  
 \sup_{v \in [0,\delta] \cup [1-\delta,1]} \big|\xi_{jn}^{h_{j},\htilde_{j}}(v) - v \big| 
\\
\label{eq: ineq for Cj minus Cj for uj close to the edge}
&\leq 2\,\sup_{C_{n}^{h,\htilde} \in \mathcal{D}_{n}^{(1)}}\,
 t_{n} \Big\{\sup_{v \in [0,\delta]\cup [1-\delta,1]} \big|h_{j}(v)\big| 
+ \sup_{v \in [0,2\,\delta]\cup [1-2\delta,1]} \big|\htilde_{j}(v)\big|  \Big\} 
 + o(t_{n}), 
\end{align}
where (similarly as above) we have used that by Lemma~\ref{lemma HD of general inverse second}(i)  
for all sufficiently large~$n$ one has 
\[
 \sup_{C_{n}^{h,\htilde} \in \mathcal{D}_{n}^{(1)}} 
 \xi_{jn}^{h_{j},\htilde_{j}}(\delta) \leq 2\delta 
 \quad \& \quad 
  \inf_{C_{n}^{h,\htilde} \in \mathcal{D}_{n}^{(1)}} \xi_{jn}^{h_{j},\htilde_{j}}(1-\delta) \geq 1-2\delta.  
\]
Now by Remark~\ref{rem: asympt diminish at borders} one can make 
$\sup_{h \in \mathcal{A}_{n}}\sup_{v \in [0,\delta]\cup [1-\delta,1]} \big|h_{j}(v)\big|$ arbitrarily small 
by taking $n$ sufficiently large and $\delta$ small enough. Note that from the properties of $\mathcal{B}$ 
and~\eqref{eq: r is diminishing at borders} 
the same is true also for 
$\sup_{v \in [0,2\,\delta]\cup [1-2\delta,1]} \big|\htilde_{j}(v)\big|$ 
as one can bound
\[
 \sup_{\htilde_{j} \in \mathcal{B}_{1}} 
 \sup_{v \in [0,2\,\delta]\cup [1-2\delta,1]} \big|\htilde_{j}(v)\big| 
 \leq \sup_{v \in [0,2\,\delta]\cup [1-2\delta,1]} r(v). 
\]

Thus combining \eqref{eq: ineq for Cj minus Cj for uj close to the edge} 
with \eqref{eq: ineq for Cj minus Cj for uj far from edge} yields that 
also the right-hand side of~\eqref{eq: C at xi minus C at u} is $o(t_{n})$ \
(uniformly in $h$, $\widetilde{h}$ and $\ub$) and so we have showed that  
\begin{equation*}
 C\big(\xib_{n}^{h,\widetilde{h}}(\ub)\big)
 - C(\ub)
 = \sum_{j=1}^{d} C^{(j)}(\ub) \big[\xi_{jn}^{h_{j},\htilde_{j}}(u_{j}) - u_{j} \big]  
 + o(t_{n}).  
\end{equation*}
This combined with Lemma~\ref{lemma HD of general inverse second}(ii) implies 
\begin{equation*}
 C\big(\xib_{n}^{h,\widetilde{h}}(\ub)\big) 
 - C(\ub)
 = t_{n}\,\sum_{j=1}^{d} C^{(j)}(\ub) \big[ - h_{j}(u_j) - \htilde_{j}\big(\xi_{jn}^{h_{j},\htilde_{j}}(u_j)\big)\big] 
 + o(t_{n}), 
\end{equation*}
which together with \eqref{eq: expanding copula functional} gives the 
statement of Theorem~\ref{thm HD of copula fctional}. 

\subsection{Proof of Theorem~\ref{thm HD of copula fctional second}}

Similarly as in the proof of Theorem~\ref{thm HD of copula fctional} 
denote
\begin{equation*} %
 \xi_{jn}^{h_{j},\htilde_{j}}(u) = \big(U + t_{n}\,h_{j} + \ttilde_{n}\,\htilde_{j}\big)^{-1}(u).  
\end{equation*}
and 
\begin{equation*} %
 \xib_{n}^{h,\widetilde{h}}(\ub) = \big(\xi_{1n}^{h_{1},\widetilde{h}_{1}}(u_{1}), \dotsc,\xi_{dn}^{h_{d},\widetilde{h}_{d}}(u_{d})\big)\tr. 
\end{equation*}

Note that Lemma~\ref{lemma HD of general inverse second}(i) implies that 
$\xib_{n}^{h,\widetilde{h}}(\ub) \ntoinfty \ub$ (uniformly in 
$C_{n}^{h,\htilde} \in \mathcal{D}_{n}^{(2)}$ 
and $\ub \in [0,1]^{d}$, $j \in \{1,\dotsc,d\}$). 
Now in the same way as in the proof of Theorem~\ref{thm HD of copula fctional}  
one can make use of the asymptotic uniform equicontinuity of the functions in $\mathcal{A}_{n}$
and the form of~$\widetilde{h}$ given in~\eqref{eq: tilde hn for copula2} to derive 
that %
 \begin{align}
\notag
  \Phi(C + t_{n}\,h + \ttilde_{n}\,\widetilde{h})(\ub) 
 &= \big(C + t_{n}\,h + \ttilde_{n}\,\widetilde{h}\big)
\big(\xib_{n}^{h,\widetilde{h}}(\ub)\big)
\\
\label{eq: expanding copula functional second}
 &= C\big(\xib_{n}^{h,\widetilde{h}}(\ub)\big) 
  + t_{n}\,h(\ub)  
  + \ttilde_{n}\, \sum_{j=1}^{d} C^{(j)}(\ub)\,\htilde_{j}\big(\xi_{jn}^{h_{j},\htilde_{j}}(u_j)\big)
   + o(t_{n}),   
 \end{align}
where here (as well as in the sequel) by $o(t_{n})$ we understand a remainder term that 
may depend on~$\ub$, $h$ 
and $\widetilde{h}$ but it is uniformly small, i.e.  
\[
\sup_{C_{n}^{h,\htilde} \in \mathcal{D}_{n}^{(2)}}
 \sup_{\ub \in [0,1]^{d}} \frac{|o(t_{n})|}{t_{n}} \ntoinfty 0. 
\]

Using the second order Taylor expansion we obtain
\begin{align}
 \notag 
 \Big|C&\big(\xib_{n}^{h,\widetilde{h}}(\ub)\big)
- C(\ub) 
 - \sum_{j=1}^{d} C^{(j)}(\ub) \big[\xi_{jn}^{h_{j},\htilde_{j}}(u_{j}) - u_{j} \big]  
 \Big|  
\\
\label{eq: C at xi minus C at u second}
 & \leq  
\sum_{j=1}^{d} \sum_{k=1}^{d} \big|C^{(j,k)}(\ub_{n}^{h,\widetilde{h}}) \big|\,  \big|\xi_{jn}^{h_{j},\htilde_{j}}(u_{j}) - u_{j} \big|\, 
\big|\xi_{kn}^{h_{k},\htilde_{k}}(u_{k}) - u_{k} \big|,      
\end{align}
where $\ub_{n}^{h,\widetilde{h}}$ lies between the points 
$\xib_{n}^{h,\htilde}(\ub)$ and $\ub$. 

Now note that if $\beta > 0$ then 
one can use 
Lemma~\ref{lemma HD of general inverse second}(iii) to deduce that 
for all sufficiently large~$n$ the $j$-th component of $\xib_{n}^{h,\htilde}(\ub)$ 
satisfies 
\begin{equation} \label{eq: ineq for xi at u}
\xi_{jn}^{h_j,\htilde_j}(u_{j}) \geq \tfrac{u_{j}}{2}, 
 \quad \& \quad 
 \xi_{jn}^{h_j,\htilde_j}(1-u_{j}) \leq 1- \tfrac{u_{j}}{2}
\end{equation}
for each $u_{j} \in [\epsilon\,t_{n}^{\vartheta},1-\epsilon\,t_{n}^{\vartheta}]$, $j \in \{1,\dotsc,d\}$, $C_{n}^{h,\htilde} \in \mathcal{D}_{n}^{(2)}$. 
And thus the same is also true for the components of~$\ub_{n}^{h,\widetilde{h}}$. 

Now with the help of Lemma~\ref{lemma HD of general inverse second}(ii), assumption~\ref{assump: copula2}, 
\eqref{eq: C at xi minus C at u second} and \eqref{eq: ineq for xi at u}
 we have 
for $\ub\in \widetilde{I}_n(\epsilon)$ that
\begin{align}
\notag 
 \Big| C\big(\xib_{n}^{h,\widetilde{h}}&(\ub)\big)
- C(\ub) 
 - \sum_{j=1}^{d} C^{(j)}(\ub) \big[\xi_{jn}^{h_{j},\htilde_{j}}(u_{j}) - u_{j} \big]  
 \Big|  
\\
\label{eq: C at xi minus C at u expansion}
& \leq \sum_{j=1}^{d} \sum_{k=1}^{d} O\big(\tfrac{1}{u_{j}^{\beta}(1-u_{j})^{\beta}
 \,u_{k}^{\beta}(1-u_{k})^{\beta}}\big)\,  
 \Big[t_{n}\,|h_{j}(u_{j})|  + 
 \ttilde_{n}\big|\htilde_{j}\big(\xi_{jn}^{h_{j},\widetilde{h}_{j}}(u_{j})\big)\big|+o(t_{n})\Big]
\\
\notag 
& \qquad  \qquad \qquad \qquad \qquad \qquad \qquad \qquad \times \Big[t_{n}\,|h_{k}(u_{k})|  + 
 \ttilde_{n}\big|\htilde_{k}\big(\xi_{kn}^{h_{k},\widetilde{h}_{k}}(u_{k})\big)\big|+o(t_{n})\Big] 
\end{align}

Now let $u_{0} \in [\epsilon\,t_{n}^{\vartheta}, 1/2]$ be fixed. Note that for each 
$j \in \{1,\dotsc,d\}$ 
\begin{align*}
 \sup_{u \in [\epsilon\,t_{n}^{\vartheta}, 1-\epsilon\,t_{n}^{\vartheta}]} 
 \frac{t_{n}\,|h_{j}(u)|}{u^{\beta}(1-u)^{\beta}} 
 & \leq O(t_{n}) + \sup_{u \in [\epsilon\,t_{n}^{\vartheta}, u_{0}] \cup [1-u_{0}, 1-\epsilon\,t_{n}^{\vartheta}]} \frac{t_{n}\,|h_{j}(u)|}{u^{\beta}(1-u)^{\beta}} 
 \\
 & \leq O(t_{n}) + O(t_{n}^{1-\vartheta \beta})\sup_{u \in [\epsilon\,t_{n}^{\vartheta}, u_{0}] \cup [1-u_{0}, 1-\epsilon\,t_{n}^{\vartheta}]} |h_{j}(u)|. 
\end{align*}
As $u_{0}$ can be taken arbitrarily small (for $n$ sufficiently large enough) one can 
conclude that 
\begin{equation*} %
 \max_{j\in\{1,\dotsc,d\}} \sup_{u \in [\epsilon\,t_{n}^{\vartheta}, 1-\epsilon\,t_{n}^{\vartheta}]} 
 \frac{t_{n}\,|h_{j}(u)|}{u^{\beta}(1-u)^{\beta}} = O(t_{n}) + o(t_{n}^{1-\vartheta \beta}).  
\end{equation*}
This together with \eqref{eq: C at xi minus C at u expansion}, \eqref{eq: set B1alpha} and  $\widetilde t_n=o(t_n^{1/2+2\gamma})$ implies 
\begin{multline} 
\label{eq: C at xi minus C at u final} 
  \sup_{\ub\in \widetilde{I}_n(\epsilon)} \Big| C\big(\xib_{n}^{h,\widetilde{h}}(\ub)\big)
- C(\ub) 
 - \sum_{j=1}^{d} C^{(j)}(\ub) \big[\xi_{jn}^{h_{j},\htilde_{j}}(u_{j}) - u_{j} \big]  
 \Big| 
 \\ 
 = O(t_{n}^{2})  + o(t_{n}^{2-2\vartheta \beta}) + o(t_{n}^{3/2-\beta\vartheta + 2\gamma - \vartheta(\beta-\alpha)_{+}})
  + o(t_{n}^{1+4\gamma - 2\vartheta (\beta-\alpha)_{+}})
  = o(t_{n}),
\end{multline}
where we use that $\beta \leq \frac{1}{2}$, $\vartheta \leq 1$ 
and $2\gamma - \vartheta (\beta - \alpha)_{+} \geq 0 $ which 
can be  deduced in the same way as in the reasoning following~\eqref{eq: B2n neglig} below by taking $s=\infty$. 

Now \eqref{eq: C at xi minus C at u final} combined 
with Lemma~\ref{lemma HD of general inverse second}(ii) implies 
\begin{equation*}
 C\big(\xib_{n}^{h,\widetilde{h}}(\ub)\big) 
 - C(\ub)
 = - t_{n}\,\sum_{j=1}^{d} C^{(j)}(\ub) \, h_{j}(u_j)
  - \ttilde_{n}\,\sum_{j=1}^{d} C^{(j)}(\ub) \, \htilde_{j}\big(\xi_{jn}^{h_{j},\htilde_{j}}(u_j)\big)
 + o(t_{n}), 
\end{equation*}
which together with \eqref{eq: expanding copula functional second} gives the 
statement of Theorem~\ref{thm HD of copula fctional second}. 

\subsection{Hadamard differentiability of a quantile function} 
\label{subsec HD of quantile}

Now it remains to generalize the result on the inverse of the cdf on~$[0,1]$. 
More precisely we are interested in the differentiable properties 
of the mapping 
\[
 \Lambda: F \to F^{-1}, \quad \text{at} \quad  F=U, 
\]
for  cumulative distribution functions 
defined on $[0,1]$ that satisfy $F(0)=0$ and $F(1)=1$. 

For the special case $d=1$ denote $\mathcal{D}_{1n}$ and  
$\mathcal{A}_{1n}$ the sets $\mathcal{D}_{n}$ and $\mathcal{A}_{n}$ 
introduced in Section~\ref{sec: HD}. Note that the elements of the set 
$\mathcal{A}_{1n}$ are uniformly bounded, asymptotically uniformly equi-continuous and satisfy the property described in Remark~\ref{rem: asympt diminish at borders}.

Further given sequence $\{t_{n}\}$ and $\{\ttilde_{n}\}$ of positive constants going to zero denote 
\begin{align*}
 \mathcal{F}_{1n} &= \{ F \in \mathcal{D}_{1n}: F = U+t_{n}\,h + \ttilde_{n}\,\widetilde{h}, 
 \ h \in \mathcal{A}_{1n}, \widetilde{h} \in \widetilde{\mathcal{H}}_{M} \},
 \\
\mathcal{F}_{1n}^{\alpha} &= \{ F \in \mathcal{D}_{1n}: F = U+t_{n}\,h + \ttilde_{n}\,\widetilde{h}, 
 \ h \in \mathcal{A}_{1n}, \widetilde{h} \in \mathcal{B}_{1n}^{\alpha} \}, 
\end{align*}
where $\widetilde{\mathcal{H}}_{M}$ and $\mathcal{B}_{1n}^{\alpha}$ 
were introduced in \eqref{eq: set HM} and \eqref{eq: set B1alpha}. Note that 
$\mathcal{F}_{1n}^{\alpha} \subset \mathcal{F}_{1n}$. Thus the statements of the following 
lemma that holds for $\mathcal{F}_{1n}$ are automatically also true for $\mathcal{F}_{1n}^{\alpha}$.

Finally denote 
\[
 \xi_{n}^{h,\widetilde{h}}(u) 
 = (U_{n}^{h,\htilde})^{-1}(u), 
 \quad u \in [0,1],    
\]
with $U_{n}^{h,\htilde} = U+t_{n}\,h + \ttilde_{n}\,\widetilde{h}$  and the inverse should be understood as in~\eqref{eq: generalised inverse}. 
\begin{lemma} \label{lemma HD of general inverse second}
Let $\mathcal{F}_{1n}$ and  $\mathcal{F}_{1n}^{\alpha}$ be as explained above. Then the following statements hold. 
\begin{enumerate}
 \item   $
 \sup_{U_{n}^{h,\htilde}\in \mathcal{F}_{1n}}\, 
 \sup_{u \in [0,1]} \big|\xi_{n}^{h,\widetilde{h}}(u) - u\big| \ntoinfty 0$. 
\item 
\begin{equation*} %
\sup_{U_{n}^{h,\htilde}\in \mathcal{F}_{1n}}\, 
\sup_{u \in [0,1]}
\bigg|\frac{\xi_{n}^{h,\widetilde{h}}(u) - u}{t_{n}} + h(u) 
+ \frac{\ttilde_{n}}{t_{n}}\,\widetilde{h}\big(\xi_{n}^{h,\widetilde{h}}(u)\big)  \bigg| 
 \ntoinfty 0.  
\end{equation*}
\item Suppose that $\ttilde_{n} = o(t_{n}^{1/2 + 2\gamma})$.  
Then for each $\epsilon > 0$ for all sufficiently large $n$  
for all $u \in [\epsilon\,t_{n}^{\vartheta},\frac{1}{2}]$: 
\[
 \sup_{U_{n}^{h,\htilde}\in \mathcal{F}_{1n}^{\alpha}}\,  \xi_{n}^{h,\htilde}\big(u\big) \leq 2u
 \quad \& \quad 
  \inf_{U_{n}^{h,\htilde}\in \mathcal{F}_{1n}^{\alpha}}\,  \xi_{n}^{h,\htilde}\big(1-u\big) \geq 1-2u
\]
and also 
\[
 \sup_{U_{n}^{h,\htilde}\in \mathcal{F}_{1n}^{\alpha}}\, 
 \xi_{n}^{h,\htilde}\big(1-u\big) \leq 1-\tfrac{u}{2} 
 \quad \& \quad 
  \inf_{U_{n}^{h,\htilde}\in \mathcal{F}_{1n}^{\alpha}}\, 
\xi_{n}^{h,\htilde}\big(u\big) \geq \tfrac{u}{2}.   
\]
\end{enumerate}
\end{lemma}
\begin{proof} 

For simplicity for an arbitrary non-decreasing function~$g$ defined on~$[0,1]$ introduce the notation 
\[
 g(v_{-}) = 
 \begin{cases} 
  \lim_{\epsilon \to 0_+} g(v - \epsilon), & \text{if } v \in (0,1], \\
   g(0), & \text{if } v = 0.                    
 \end{cases}  
\]
Using this notation by the definition of the generalized inverse function 
\begin{equation*} %
 \big(U+t_{n}\,h+\ttilde_{n}\,\widetilde{h}\big)\big(\xi_{n}^{h,\widetilde{h}}(u)_{-}\big) 
 \leq u 
  \leq \big(U+t_{n}\,h+\ttilde_{n}\,\widetilde{h}\big)\big(\xi_{n}^{h,\widetilde{h}}(u)\big),    
\end{equation*}
which with the help of \eqref{eq: Dn} yields that 
\begin{equation} \label{eq: ineq for quantiles second}
  \xi_{n}^{h,\widetilde{h}}(u) + t_{n}\,h\big(\xi_{n}^{h,\widetilde{h}}(u)\big) +\ttilde_{n}\,\widetilde{h}\big(\xi_{n}^{h,\widetilde{h}}(u)\big) 
  + o(t_{n})
  \leq u \leq \xi_{n}^{h,\widetilde{h}}(u) + t_{n}\,h\big(\xi_{n}^{h,\widetilde{h}}(u)\big) + \ttilde_{n}\,\widetilde{h}\big(\xi_{n}^{h,\widetilde{h}}(u)\big). 
\end{equation}
From this one can conclude that
\begin{equation*}
  \sup_{U_{n}^{h,\htilde}\in \mathcal{F}_{1n}}\, 
\sup_{u \in [0,1]} \big|\xi_{n}^{h,\widetilde{h}}(u) - u\big| 
 \leq t_{n} \sup_{h \in \mathcal{A}_{1n}}\, 
\sup_{u \in [0,1]} \big|h(u) \big| +
\ttilde_{n}
\sup_{\widetilde{h} \in \widetilde{\mathcal{H}}_{M}}\, 
\sup_{u \in [0,1]} \big|\widetilde{h}(u) \big| 
+ o(t_{n})
 \ntoinfty 0,
\end{equation*}
which yields \textit{the first statement} of the lemma. 

\medskip 

To prove \textit{the second statement} note that the inequalities 
in~\eqref{eq: ineq for quantiles second} can be rewritten as 
\begin{align*} %
 - h\big(\xi_{n}^{h,\widetilde{h}}(u)\big) + h(u) + o(1) 
&  \leq \frac{\xi_{n}^{h,\widetilde{h}}(u) - u}{t_{n}} + h(u) + \frac{\ttilde_{n}}{t_{n}}\,\widetilde{h}\big(\xi_{n}^{h,\widetilde{h}}(u)\big) 
\\
& \leq  - h\big(\xi_{n}^{h,\widetilde{h}}(u)_{-}\big) + h(u) 
+ o(1).  
\end{align*}

Now using the first statement of the lemma 
and that the functions in $\mathcal{A}_{1n}$ are 
asymptotically equicontinuous (see Section \ref{sec: HD}) one can conclude 
that both the left-hand side and the right-hand side of the above inequalities  
converge to zero (uniformly in $u$, $h$ and $\widetilde{h}$)  
which was to be proved. 

\medskip 

Regarding \textit{the third statement} note that 
for each fixed positive $u_{0}$ the proof for 
$u \in [u_{0}, \tfrac{1}{2}]$ follows from the first statement 
of the lemma. 

Thus in what follows we prove that 
for all $u \in [\epsilon\,t_{n}^{\vartheta},u_{0}]$
\begin{equation} \label{eq: ineq for quantiles to be proved}
 \sup_{U_{n}^{h,\htilde}\in \mathcal{F}_{1n}^{\alpha}}\,  \xi_{n}^{h,\htilde}(u) \leq 2 u
\quad \text{ and } \quad 
  \inf_{U_{n}^{h,\htilde}\in \mathcal{F}_{1n}^{\alpha}}\, 
\xi_{n}^{h,\htilde}(u) \geq \tfrac{u}{2}.   
\end{equation}
The remaining statements can be proved analogously. 

To prove the first inequality in \eqref{eq: ineq for quantiles to be proved} 
suppose that for some $n\in \mathbb{N}$ there exists 
$u \in [\epsilon\,t_{n}^{\vartheta},u_{0}]$ and $U_{n}^{h,\htilde}\in \mathcal{F}_{1n}^{\alpha}$ 
such that 
\[
 \xi_{n}^{h,\htilde}(u) > 2u. 
\]
This implies that 
\begin{align*}
 u & >  \big(U + t_{n}\,h + \ttilde_{n}\,\htilde \big)(2u)  
 = 2u + t_{n}\,h(2u) + 
 \ttilde_{n}\,\htilde(2u)
\\
&\geq 2 u + t_{n}\,h(2u) 
 + o(t_{n}^{1/2 + 2 \gamma})\,u^{\alpha}(1-u)^{\alpha}
\end{align*}
which further gives 
\begin{equation} \label{eq: u less than contradiction}
  u  \leq  - t_{n}\,h(2u) 
  - o(t_{n}^{1/2+2\gamma})\,u^{\alpha}.
\end{equation}
Note that~\eqref{eq: assump on gamma and vartheta} 
implies that $\vartheta \leq \frac{1/2+2\gamma}{1-\alpha}$ and thus 
\begin{align*}
1 &\leq  \big(\tfrac{t_{n}}{u}\big)\, |h(2u)| 
  + o\Big(\tfrac{t_{n}^{1/2+2\gamma}}{u^{1-\alpha}}\Big) 
\\
&\leq O(t_{n}^{1-\vartheta})\,\sup_{u \in [0,2u_{0}]} |h(u)| 
 + o\big(t_{n}^{1/2+2\gamma - \vartheta (1-\alpha)}\big)  
= O(1)\,\sup_{u \in [0,2u_{0}]} |h(u)| 
 + o(1).   
\end{align*}
But this is a contradiction as by Remark~\ref{rem: asympt diminish at borders} 
the right hand side of the last inequality  
can be made arbitrarily small uniformly in~$h$ by taking $u_{0}$ small enough.

Analogously to prove the second inequality in \eqref{eq: ineq for quantiles to be proved} 
suppose that for some $u \in [\epsilon\,t_{n}^{\vartheta},\delta]$ 
and $U_{n}^{h,\htilde}\in \mathcal{F}_{1n}^{\alpha}$ 
\[
 \xi_{n}^{h,\htilde}\big(u\big) < \tfrac{u}{2}.   
\]
This implies that 
\begin{equation} \label{eq: u less than}
 u \leq \big(U + t_{n}\,h + \ttilde_{n}\,\htilde \big)(\tfrac{u}{2})  
 = \tfrac{u}{2} + t_{n}\,h(\tfrac{u}{2}) + o(t_{n}^{1/2+2\gamma})\,\htilde(\tfrac{u}{2}). 
\end{equation}
Note that by the properties of $\mathcal{B}_{1n}^{\alpha}$ 
one has $\htilde(\frac{u}{2})  = O(u^{\alpha})$.  
Thus with the help of~\eqref{eq: u less than} 
\[
 \tfrac{u}{2} \leq   t_{n} 
\sup_{v \in [0,\frac{\delta}{2}]}
 |h(v)| 
 + o(t_{n}^{1/2+2\gamma})\,O(u^{\alpha}).
\]
Now one can arrive at a contradiction similarly as from 
the inequality~\eqref{eq: u less than contradiction}.
\end{proof}

\section{Proofs of results in Section~\ref{sec: applic to empir resid copula}} \label{appendix B}

\subsection{Restricting to a subset of \texorpdfstring{$[0,1]^d$}{[0,1]d}}
 
Let $\epsilon > 0$ be fixed. First of all we show that it is sufficient to consider 
$\ub \in J_{n}(\epsilon)$, where 
\begin{equation} \label{eq: Jn epsilon}
J_{n}(\epsilon) = \big[\tfrac{\epsilon}{\sqrt{n}},1-\tfrac{\epsilon}{\sqrt{n}}\big]^{d}.  
\end{equation}
Note that 
\[
 [0,1]^{d} \setminus J_{n}(\epsilon) = \bigcup_{j=1}^{d} \big(J_{jn}^{(L)}(\epsilon) \cup  J_{jn}^{(U)}(\epsilon) \big),  
\]
where 
\[
 J_{jn}^{(L)}(\epsilon) = \big\{ \ub \in [0,1]^{d} : u_{j} < \tfrac{\epsilon}{\sqrt{n}}\,\big\}, 
 \qquad 
 J_{jn}^{(U)}(\epsilon) = \big\{ \ub \in [0,1]^{d} : u_{j} > 1 - \tfrac{\epsilon}{\sqrt{n}}\,\big\}.  
\]
Suppose for a moment that \underline{$\ub \in J_{jn}^{(L)}(\epsilon)$}, then 
\[
 \big|\sqrt{n}\,\big[\widehat{C}_{n}(\ub) - C(\ub)\big]\big| \leq 
 \sqrt{n}\,\big[|\widehat{C}_{n}(\ub)| + |C(\ub)|\big] 
 \leq \sqrt{n}\, \big[\widehat{F}_{j\epshat}\big(\widehat{F}_{j\epshat}^{-1}(u_j)\big) + u_{j}\big]
 \leq 2\,\epsilon +\tfrac{1}{\sqrt{n}}. 
\]

Now consider that \underline{$\ub \in J_{jn}^{(U)}(\epsilon)$}. 
Denote $\ub^{(-j)}$ the vector~$\ub$ whose $j$-th component is  
replaced with~$1$, i.e.
\[
 \ub^{(-j)} = (u_1,\dotsc,u_{j-1},1,u_{j+1},\dotsc,u_{d})\tr.  
\]
Then from the basic properties of the copula function \citep[see e.g.\ Theorem~2.2.4][]{nelsen_2006}
\begin{equation} \label{eq: comparing C at two poins with large uk}
 \sqrt{n}\,\big|C(\ub^{(-j)}) -  C(\ub)\big| \leq \epsilon. 
\end{equation}
Similarly also 
\begin{align} 
 \sqrt{n}\,\big|\widehat{C}_{n}(\ub^{(-j)}) - \widehat{C}_{n}(\ub)\big| 
 & = 
 \frac{1}{\sqrt{n}} \suman 
 \Big[\ind\big\{ \epshat_{1i} \leq \widehat{F}_{1\epshat}^{-1}(u_1), 
 \dotsc,  \epshat_{ji} \leq \widehat{F}_{j\epshat}^{-1}(1) , \dotsc,  \epshat_{di} \leq \widehat{F}_{1\epshat}^{-1}(u_d) 
 \big\} 
\notag \\ 
 & \qquad \qquad \  - \ind\big\{ \epshat_{1i} \leq F_{1\epshat}^{-1}(u_1), 
 \dotsc, \epshat_{ji} \leq F_{j\epshat}^{-1}(u_j) , \dotsc, \epshat_{di} \leq F_{d\epshat}^{-1}(u_d) 
 \big\} \Big]
\notag  \\
& \leq
 \frac{1}{\sqrt{n}} \suman 
 \Big[\ind\big\{ \epshat_{ji} \leq \widehat{F}_{j\epshat}^{-1}(1) \big\} 
   - \ind\big\{ \epshat_{ji} \leq \widehat{F}_{j\epshat}^{-1}(u_{j}) \big\} \Big]
\notag \\
\label{eq: comparing Cnhat at two poins with large uk}
 &= \sqrt{n}\,\big[1 - \Fhat_{j\epshat}\big(\Fhat_{j\epshat}^{-1}(u_j)\big)\big] 
  \leq \epsilon. 
\end{align}
Now combining \eqref{eq: comparing C at two poins with large uk} and 
\eqref{eq: comparing Cnhat at two poins with large uk} yields that 
uniformly in $\ub \in J_{jn}^{(U)}$
\[
 \sqrt{n}\,\big[\widehat{C}_{n}(\ub) - C(\ub)\big] 
  = \sqrt{n}\,\big[\widehat{C}_{n}(\ub^{(-j)}) - C(\ub^{(-j)})\big] + O(\epsilon).    
\]
Repeatedly using this argument for other components of $\ub$ bigger than 
$1 - \tfrac{\epsilon}{\sqrt{n}}$ one can conclude that without loss 
of generality one can consider only 
$\ub \in J_{n}(\epsilon)$. 
 
\bigskip

Note that to prove Theorem~\ref{thm for nonparam adjustment} 
it is sufficient to consider $\ub$ from the set 
\begin{equation*} %
\Jtilde_{n}(\epsilon) = \big[\tfrac{\epsilon}{n^{\vartheta/2}},1-\tfrac{\epsilon}{n^{\vartheta/2}}\big]^{d}.  
\end{equation*}
which is for $\vartheta < 1$ a (strict) subset of $J_{n}(\epsilon)$, see \eqref{eq: Jn epsilon}. Thus denote 
\begin{equation*} %
  \Jring_{n}(\epsilon) = 
 \begin{cases}
   J_{n}(\epsilon), \text{when proving Theorem~\ref{thm for param adjustment}}, \\
   \Jtilde_{n}(\epsilon), \text{when proving Theorem~\ref{thm for nonparam adjustment}}.  \\
 \end{cases}
\end{equation*}

\subsection{Using Theorems~\ref{thm HD of copula fctional} and~\ref{thm HD of copula fctional second}}

Let $\widehat{G}_{n\epshat}$ be as in~\eqref{eq: Ghat epshat} 
where the residuals are of the general form~\eqref{eq: resid general}. 
We will show in Section \ref{subsec: dealing with An} that for each $\epsilon > 0$ 
the representation~\eqref{eq: reprentation of Geps} 
holds uniformly in $\ub \in \Jring_{n}(\epsilon)$ 
with $\AAA_{n}$, $\mathbb{B}_{n}$ given by~\eqref{eq: An typical}, 
\eqref{eq: Bn typical} and $\ZZZ_{jn}$ introduced in~\eqref{eq: Zjn}. 

\smallskip 

First note that the marginals of $\widehat{G}_{n\epshat}$ 
are discontinuous with jumps of the height~$\frac{1}{n}$. Thus 
for $t_{n}=\frac{1}{\sqrt{n}}$ there exists a sequence of 
functions $\mathcal{D}_{n}$ such that even 
$\PP\big(\widehat{G}_{n\epshat} \in \mathcal{D}_{n}\big)=1$. 

Further, as explained in Section~\ref{sec: HD} for each $\eta > 0$ there exists   a sequence of sets $\{\mathcal{A}_{n}\}$ of functions on $[0,1]^{d}$ such that \eqref{eq: An in Kn} holds and 
Theorem~\ref{thm HD of copula fctional} is satisfied.

Finally if \ref{assump: process Z strict} (or \ref{assump: process Z less strict}) 
hold then  
as explained in Remark~\ref{rem: Zn in view of assumptions}  
there exists a set $\mathcal{B}$ as in (\ref{eq: tilde hn})
(or a sequence of sets $\{\mathcal{B}_{n}^{\alpha}\}$ as in (\ref{eq: tilde hn for copula2})) of functions on $[0,1]^{d}$
such that assumptions of Theorem~\ref{thm HD of copula fctional} 
(or Theorem~\ref{thm HD of copula fctional second}) 
with $t_{n} = \frac{1}{\sqrt{n}}$ 
(and $\widetilde{t}_{n} = n^{-(1/4+\gamma)}$ ) 
are met 
and at the same time 
\[
 \liminf_{n \to \infty} \PP\big(\mathbb{B}_{n} \in \mathcal{B}\big) 
 \geq 1- \eta 
\quad \Big(\text{or} \quad 
  \liminf_{n \to \infty} \PP\Big(\tfrac{t_n}{\widetilde t_n}\mathbb{B}_{n} \in \mathcal{B}_{n}^{\alpha}\Big) 
 \geq 1- \eta \Big).  
\]
Now note that one can rewrite the empirical copula process as 
\[
 \sqrt{n}\,\big(\widehat{C}_{n} - C\big) = \frac{\Phi(\widehat{G}_{n\epshat}) - \Phi(C)}{\tfrac{1}{\sqrt{n}}}, 
\]
where $\Phi$ stands for the copula mapping formally introduced in Section~\ref{sec: HD}. 
Thus with the help of %
\eqref{eq: reprentation of Geps}  
one can use Theorem~\ref{thm HD of copula fctional} 
(or Theorem~\ref{thm HD of copula fctional second}) 
with $h=\AAA_{n}$ and $\widetilde{h} = n^{-1/4+\gamma}\BBB_{n}=\frac{t_n}{\widetilde t_n}\BBB_{n}$ 
to deduce that 
(uniformly in $\ub \in \Jring_{n}(\epsilon)$)
\begin{equation}
\sqrt{n}\,\big[\widehat{C}_{n}(\ub) - C(\ub)\big] 
\\
\label{eq: as repr for Cnhat proved}
=  \AAA_{n}(\ub) - \sum_{j=1}^{d} C^{(j)}(\ub)\AAA_{n}(\ub^{(j)})
  + o_{P}(1).   
\end{equation}
Now 
the standard result for the empirical process copula 
\citep[see e.g.\ Proposition~3.1][]{segers_empirical_2012} together with 
\eqref{eq: as repr for Cnhat proved} implies that 
asymptotic equivalence \eqref{eq: equiv of Cn and oracle Cn}
holds for each $\epsilon > 0$ 
uniformly in $\ub \in \Jring_{n}(\epsilon)$ 
which yields the statement of the theorem.  

\subsection{Proof of~\texorpdfstring{\eqref{eq: reprentation of Geps}}{the asymptotic representation Gnhat}} 
\label{subsec: dealing with An}
Introduce 
\begin{equation} \label{eq: process Ghnhat}
\widehat{\mathbb{G}}_{n}(\ub) = 
  \sqrt{n}\,\big[\hatG_{n\hat\varepsilon}(\ub) - C(\ub)\big]  
\end{equation}
with $\hatG_{n\hat\varepsilon}$ from (\ref{eq: Ghat eps})
and note that \eqref{eq: reprentation of Geps} is equivalent 
to showing that
\begin{equation*} %
 \widehat{\mathbb{G}}_{n}(\ub) 
= \AAA_{n}(\ub)   + \mathbb{B}_{n}(\ub) .
\end{equation*}
Now recall the set of functions $\mathcal{F}$ introduced in 
\eqref{eq: set of functions F}. 
Note that each function in $\mathcal{F}$ can be identified 
with `parameter' $\bigtimes_{j=1}^{d} (z_{j}, \ttilde_{j})$ 
(having $2d$~components). Now 
$\widehat{\mathbb{G}}_{n}(\ub)$ given 
by~\eqref{eq: process Ghnhat} can be rewritten with the help 
of the standard empirical processes notation as 
\begin{equation} \label{eq: Gnhat as empir process}
\widehat{\mathbb{G}}_{n}(\ub) = 
  \sqrt{n}\,\big[P_{n}(f_{n}^{(\ub)} - f_{0}^{(\ub)})\big] 
 + \sqrt{n}\,\big[P_{n}(f_{0}^{(\ub)}) - P(f_{0}^{(\ub)})\big] ,
\end{equation} 
where $P$ stands for the expectation given by the distribution of $\binom{\Yb}{\Xb}$, 
$P_{n}$ for the empirical expectation given by the observed data 
$\binom{\Yb_{1}}{\Xb_{1}},\dotsc,\binom{\Yb_{n}}{\Xb_{n}}$ and 
\[
 f_{n}^{(\ub)} = \bigtimes_{j=1}^{d} \big(\widetilde{F}_{nj\eps}^{-1}(u_{j}), \that_{j}\big), 
 \qquad 
f_{0}^{(\ub)} = \bigtimes_{j=1}^{d} \big(F_{j\eps}^{-1}(u_j), t_{j}\big). 
\]
Thus thanks to assumption~\ref{assump: donsker} we know that 
the empirical process $\sqrt{n}(P_{n} - P)$ is asymptotically 
uniformly equicontinuous in probability with respect to the semimetric 
$\rho$ given by~\eqref{eq: semimetric rho}. 
Further with the help of 
assumption~\ref{assump: process Z strict} (or~\ref{assump: process Z less strict}) 
\begin{align*} %
\sup_{\ub \in \Jring_{n}(\epsilon)}
P\big|f_{n}^{(\ub)} -f_{0}^{(\ub)}\big| 
& \leq  \sum_{j=1}^{d} 
\sup_{\ub \in \Jring_{n}(\epsilon)} 
 \E \big|\,\ind\big\{\that_{j}(Y_{j}; \Xb)  \leq \widetilde{F}_{nj\eps}^{-1}(u_{j})\big\} - \ind\big\{t_{j}(Y_{j}; \Xb)  \leq F_{j\eps}^{-1}(u_{j})\big\}\big|
\\
  & \leq 2\,\sum_{j=1}^{d} 
\sup_{\ub \in \Jring_{n}(\epsilon)} 
  \E_{\Xb}|Z_{jn}(u_{j}; \Xb)| 
  = o_{P}(1), 
\end{align*}
where the expectation on the first line is with respect to the random vector $\binom{Y_{j}}{\Xb}$  
keeping $\that_{j}$ fixed. 

Thus from the asymptotic uniform equicontinuity of the process 
$\big\{\sqrt{n}(P_{n}(f) - P(f)), f \in \mathcal{F} \big\}$
and assumption \ref{assump: true tj} one can conclude 
that uniformly in $\ub \in \Jring_{n}(\epsilon)$ 
\begin{equation*} %
   \sqrt{n}\,\big[P_{n}(f_{n}^{(\ub)} - f_{0}^{(\ub)})\big] 
 = \sqrt{n}\,\big[P(f_{n}^{(\ub)} - f_{0}^{(\ub)})\big] 
  + o_{P}(1),  
\end{equation*}
which combined with \eqref{eq: Gnhat as empir process} implies that 
(uniformly in $\ub \in \Jring_{n}(\epsilon)$)  
\begin{equation} \label{eq: Gnhat via implicit An plus Bn}
 \widehat{\mathbb{G}}_{n}(\ub) = \sqrt{n}\,\big[P_{n}(f_{0}^{(\ub)}) - P(f_{0}^{(\ub)})\big] + 
 \sqrt{n}\,\big[P(f_{n}^{(\ub)} - f_{0}^{(\ub)})\big] 
  + o_{P}(1).  
\end{equation}
Note that the first term on the right-hand side 
of~\eqref{eq: Gnhat via implicit An plus Bn} can be rewritten 
as 
\begin{align*}
 \sqrt{n}\,&\big[P_{n}(f_{0}^{(\ub)}) - P(f_{0}^{(\ub)})\big] 
 \\
  &= \frac{1}{\sqrt{n}} 
 \suman [\,\ind\{t_{1}(Y_{1i};\Xb_{i}) \leq F_{1\eps}^{-1}(u_1), 
   \dotsc, t_{d}(Y_{di};\Xb_{i}) \leq F_{d\eps}^{-1}(u_d) \} 
 - C(\ub) ]
 \\
 &= \frac{1}{\sqrt{n}} 
 \suman [\,\ind\{\eps_{1i} \leq F_{1\eps}^{-1}(u_1), 
   \dotsc, \eps_{di} \leq F_{d\eps}^{-1}(u_d) \} 
 - C(\ub) ]
= \AAA_{n}^*(\ub), 
\end{align*}
where $\AAA_{n}^*$ denotes the dominating term in the definition of $\AAA_{n}$ in \eqref{eq: An typical}. 

\medskip 

In what follows we need to explore the second term 
on the right-hand side of~\eqref{eq: Gnhat via implicit An plus Bn}  
which we denote as $\widetilde{\mathbb{B}}_{n}$. We will show that 
(uniformly in $\ub \in \Jring_{n}(\epsilon)$)  
\begin{equation} \label{eq: tilde Bn equals Bn}
 \widetilde{\mathbb{B}}_{n}(\ub) = \mathbb{B}_{n}(\ub) + o_{P}(1), 
\end{equation}
where $\mathbb{B}_{n}$ is introduced in~\eqref{eq: Bn typical}.

\subsection{Dealing with \texorpdfstring{$\widetilde{\mathbb{B}}_{n}$}{tilde Bn}}
Note that
\begin{align*}
 \notag  
 \widetilde{\mathbb{B}}_{n}(\ub) &= \sqrt{n}\,\big[P(f_{n}^{(\ub)} - f_{0}^{(\ub)})\big] 
\\
\notag 
& = \sqrt{n}\,\E_{\Xb} \bigg[F_{\epsb}\bigg(
  t_{1}\big\{\that_{1}^{-1}\big(\widetilde{F}_{n1\eps}^{-1}(u_{1}); \Xb\big); \Xb \big\},\dotsc,t_{d}\big\{\that_{d}^{-1}\big(\widetilde{F}_{nd\eps}^{-1}(u_{d}); \Xb\big); \Xb \big\}
\bigg)
\\
\notag 
& \qquad \qquad \quad  - F_{\epsb}\big(F_{1\eps}^{-1}(u_1),\dotsc,F_{d\eps}^{-1}(u_d) \big) \bigg]
\\
&= 
\sqrt{n}\,\E_{\Xb} \big[C\big(\widehat{\xib}_{n}^{\Xb}(\ub) 
\big) - C(\ub) \big], 
\end{align*}
where $F_{\epsb}$ stands for the joint distribution function 
of $\epsb = (\eps_1,\dotsc,\eps_d)\tr$ and 
\[
 \widehat{\xib}_{n}^{\xb}(\ub) 
 = \big(\widehat{\xi}_{1n}^{\xb}(u_{1}),\dotsc, \widehat{\xi}_{dn}^{\xb}(u_{d}) \big), 
 \quad  \text{with} \quad 
 \widehat{\xi}_{jn}^{\xb}(u_{j}) =
 F_{j\eps}\Big(t_{j}\big\{\that_{j}^{-1}\big(\widetilde{F}_{nj\eps}^{-1}(u_{j});
 \xb\big);\xb\big\}\Big).
\]
Now  the proof of~\eqref{eq: tilde Bn equals Bn} 
depends on whether we assume \ref{assump: copula1} and \ref{assump: process Z strict}  
or \ref{assump: copula2} and \ref{assump: process Z less strict}  . 

\smallskip 

Suppose that assumptions \ref{assump: copula1} and \ref{assump: process Z strict} hold. 
Then one can use the mean value theorem to bound 
\begin{align} 
\label{eq: Bn first order}
 \big|\widetilde{\mathbb{B}}_{n}(\ub) - \mathbb{B}_{n}(\ub) \big|
 &\leq \sqrt{n}\,\sum_{j=1}^{d} \E_{\Xb}\big|C^{(j)}(\ub_{n}^{\Xb}) - C^{(j)}(\ub)\big|\, 
 \big|Z_{jn}(u_{j};\Xb) \big|
\\
\notag 
 &\leq  O_{P}(1) \sum_{j=1}^{d} 
  r(u)\,\E_{\Xb}\,\big|C^{(j)}(\ub_{n}^{\Xb}) - C^{(j)}(\ub)\big|\, 
  M(\Xb)
  + o_{P}(1),  
\end{align}
where $\ub_{n}^{\Xb}$ lies between $\widehat{\xib}_{n}^{\xb}(\ub)$ and $\ub$.  
Now note that for each sequence $\{a_{n}\}$ going to infinity 
\[
  \E_{\Xb}\, M(\Xb) \ind\{M(\Xb) > a_{n}\} \ntoinfty 0.  
\]
Thus it is sufficient to consider 
\[
 \sum_{j=1}^{d} r(u)\,\E_{\Xb}\big|C^{(j)}(\ub_{n}^{\Xb}) - C^{(j)}(\ub)\big|\, 
  M(\Xb)\,\ind\{M(\Xb) \leq n^{1/3}\}.
\]
Further with the help of~\eqref{eq: Zjn} and assumption \ref{assump: process Z strict} 
one can conclude that for each $j \in \{1,\dotsc,d\}$
\begin{align*}
 \big|\widehat{\xi}_{jn}^{\Xb}(u) - u\big|&\, \ind\big\{M(\Xb) \leq n^{1/3}\big\} 
  =  \big|Z_{jn}(u_{j};\Xb)\big|\, \ind\big\{M(\Xb) \leq n^{1/3}\big\} 
\\
& \leq M(\Xb)\,\ind\big\{M(\Xb) \leq n^{1/3}\big\}\, 
 r(u)\, O_{P}\big(\tfrac{1}{\sqrt{n}}\big)  + o_{P}\big(\tfrac{1}{\sqrt{n}}\big) 
= o_{P}(1).   
\end{align*}
This implies that $\ub_{n}^{\Xb} = \ub + o_{P}(1)$ uniformly in 
$\ub \in J_{n}(\epsilon)$ (on the event $[M(\Xb) \leq n^{1/3}]$).
Note that 
\begin{equation*} %
 \E_{\Xb}\big|C^{(j)}(\ub_{n}^{\Xb}) - C^{(j)}(\ub)\big|\, 
  M(\Xb)\ind\{M(\Xb) \leq n^{1/3}\} 
 \leq \, \E_{\Xb} M(\Xb) < \infty.  
\end{equation*}
Thus for every $\eta > 0$  one can find $\delta > 0$ such that 
\[
\sup_{\ub \in [0,1]^{d}: u_{j}\in [0,\delta] \cup [1-\delta,1]}
 r(u)\,\E_{\Xb}\big|C^{(j)}(\ub_{n}^{\Xb}) - C^{(j)}(\ub)\big|\, 
M(\Xb)\ind\{M(\Xb) \leq n^{1/3}\} < \eta 
\]
and it remains to consider the case $u_{j} \in [\delta, 1-\delta]$. 
For this we can work conditionally 
on the event 
\[
 \sup_{\ub\in J_n(\epsilon)} \|\ub_{n}^{\Xb}-\ub\| \leq \tilde{\delta}, 
\quad  
\]
for each $\tilde{\delta} > 0$ and we require $\tilde{\delta} < \frac{\delta}{2}$.  
Now thanks to assumption~\ref{assump: copula1} one can bound
\begin{multline*} 
\sup_{\ub\in J_n(\epsilon), u_{j} \in [\delta,1-\delta]}\sum_{j=1}^{d} r(u)\,\E_{\Xb}\big|C^{(j)}(\ub_{n}^{\Xb}) - C^{(j)}(\ub)\big|\, 
M(\Xb)\ind\{M(\Xb) \leq n^{1/3}\}
\\ 
 \leq 
\sup_{\ub,\vb \in [0,1]^d: u_{j},v_{j}\in [\delta/2,1-\delta/2],
\|\ub - \vb\| \leq \tilde{\delta}} \big|C^{(j)}(\ub) - C^{(j)}(\vb)\big|
\, 
 \E_{\Xb} M(\Xb), 
\end{multline*}
which can be made arbitrarily small by taking $\tilde{\delta}$ small enough. 

This finishes the proof of Theorem~\ref{thm for param adjustment}. 

\medskip 

Suppose that assumptions~\ref{assump: copula2}, \ref{assump: process Z less strict} hold and introduce $a_{n} = n^{\frac{1/4-\gamma}{s-1}}$. Then  
similarly as in \eqref{eq: Bn first order} there exists 
$\ub_{n}^{\Xb}$  between $\widehat{\xib}_{n}^{\Xb}(\ub)$ and $\ub$ 
such that 
\begin{align} 
\notag 
 \big|\widetilde{\mathbb{B}}_{n}(\ub) - \mathbb{B}_{n}(\ub) \big|
 &\leq \sqrt{n}\,\sum_{j=1}^{d} \E_{\Xb}\big|C^{(j)}(\ub_{n}^{\Xb}) - C^{(j)}(\ub)\big|\, 
 \big|Z_{jn}(u_{j};\Xb) \big| \ind\{M(\Xb) < a_{n}\}
\\
\label{eq: Bn1 and Bn2}
 &\quad  + \sqrt{n}\,\sum_{j=1}^{d} \E_{\Xb}\big|C^{(j)}(\ub_{n}^{\Xb}) - C^{(j)}(\ub)\big|\, 
 \big|Z_{jn}(u_{j};\Xb) \big| \ind\{M(\Xb) \geq a_{n}\}
\\
\notag  
& = \mathbb{B}_{n1}(\ub) + \mathbb{B}_{n2}(\ub),  
\end{align}
where $\mathbb{B}_{n1}$ and $\mathbb{B}_{n2}$ stand 
for the first and second term on the right-hand 
side of \eqref{eq: Bn1 and Bn2}. 

Now with the help of \ref{assump: process Z less strict} 
one can bound $\mathbb{B}_{n2}$ as 
\begin{align*}
\sup_{\ub \in \tilde{J}_{n}(\epsilon)} \mathbb{B}_{n2}(\ub) 
& = \sqrt{n}\, o_{P}(n^{-1/4-\gamma}) 
 \E_{\Xb} \big[ M(\Xb) \ind\{M(\Xb) \geq a_{n}\} \big]
 \\
& = o_{P}(n^{1/4 - \gamma})  
\tfrac{\E_{\Xb} M^{s}(\Xb)}{a_{n}^{s-1}} = o_{P}(1).  
\end{align*}
Thus one can concentrate on $\mathbb{B}_{n1}$. Using 
once more the mean value theorem one gets  
\begin{align}
 \notag 
  \mathbb{B}_{n1}(\ub) 
 & \leq \sqrt{n}\,\sum_{j=1}^{d} \sum_{k=1}^{d} 
  \E_{\Xb}\,|C^{(j,k)}(\ub_{n}^{\Xb})|\, 
 \big|Z_{jn}(u_{j};\Xb) \big|
 \,\big|Z_{kn}(u_{k};\Xb) \big| \ind\{M(\Xb) < a_{n}\}
\\
\notag 
 &\leq o_{P}\big(n^{-2\gamma}\big)\,
  \sum_{j=1}^{d} \sum_{k=1}^{d} u_{j}^{\alpha}(1-u_{j})^{\alpha}\, 
 u_{k}^{\alpha}(1-u_{k})^{\alpha}
 \\
\label{eq: Bn minus tildeBn second} 
  & \qquad \qquad \qquad  \qquad \qquad  
 \E_{\Xb} \big[\big|C^{(j,k)}(\ub_{n}^{\Xb})\big|  
 \, M^{2}(\Xb) \big| \ind\{M(\Xb) < a_{n}\} \big],  
\end{align}
where $\widetilde{\ub}_{n}^{\Xb}$ lies between $\widehat{\xib}_{n}^{\Xb}(\ub)$ and $\ub$. 

Now if $\beta =0$ then \eqref{eq: tilde Bn equals Bn} follows immediately as the 
second derivatives of the copula function~$C$ are bounded. Thus 
suppose that $\beta > 0$. Then with the help of assumption~\ref{assump: process Z less strict}
uniformly in $u \in [\frac{\epsilon}{n^{\vartheta/2}},1-\frac{\epsilon}{n^{\vartheta/2}}]$ 
on the event $\{M(\Xb) < a_{n}\}$ it holds that
\begin{align*}
  \widehat{\xi}_{jn}^{\Xb}(u)    
  & = Z_{jn}(u; \Xb)     + u 
  \\ 
 & = u^{\alpha}(1-u)^{\alpha} o_{P}(a_{n}\,n^{-(1/4+\gamma)}) + u
 = u(1-u) o_{P}(1) + u,  
\end{align*}
as 
\[
 (1-\alpha)\tfrac{\vartheta}{2} + \tfrac{1/4-\gamma}{s-1} - \tfrac{1}{4} - \gamma \leq 0,  
\]
by the definition of~$\vartheta$ given in~\eqref{eq: vartheta}. 

Thus one can conclude that 
\[
 \PP\big(\tfrac{u}{2} \leq \xi_{jn}^{\Xb}(u); \forall u \in \big[\tfrac{\epsilon}{n^{\vartheta/2}},1 - \tfrac{\epsilon}{n^{\vartheta/2}}\big],  
  \forall j \in \{1,\dotsc,d\}
  \mid  M(\Xb) < a_{n}\big) 
 \ntoinfty 1
\]
and also 
\[
 \PP\big(\xi_{jn}^{\Xb}(1-u) \leq 1 - \tfrac{u}{2};  \forall u \in \big[\tfrac{\epsilon}{n^{\vartheta/2}},1 - \tfrac{\epsilon}{n^{\vartheta/2}}\big],  
  \forall j \in \{1,\dotsc,d\}
  \mid  M(\Xb) < a_{n} \big)  
 \ntoinfty 1. 
\]
Thus the same holds true also for the components of $\ub_{n}^{\xb}$. So 
with the help of assumption~\ref{assump: copula2} 
one can further rewrite \eqref{eq: Bn minus tildeBn second} as 
\begin{align}
\notag 
\mathbb{B}_{2n}(\ub) 
  &\leq 
  o_{P}(n^{-2\gamma})\,
  \sum_{j=1}^{d} \sum_{k=1}^{d} 
O\big(\tfrac{1}{u_{j}^{\beta-\alpha}(1-u_{j})^{\beta-\alpha}
 \,u_{k}^{\beta-\alpha}(1-u_{k})^{\beta-\alpha}}\big) 
 + o_{P}(1)
\\
\label{eq: B2n neglig}
& = o_{P}\big(n^{-[2\gamma - (\beta - \alpha)\vartheta]}\big) + o_{P}(1) 
 = o_{P}(1), 
\end{align}
where the last equality is implied the definition of~$\vartheta$ and 
properties of~$\gamma$ given in~\eqref{eq: vartheta} and~\eqref{eq: gamma} 
respectively as follows. Note that it is sufficient to consider $\beta > \alpha$. 
Now distinguish two cases. 

(i) First if \underline{$\vartheta = 1$} then by~\eqref{eq: vartheta}
\[
  \frac{\frac{s-2}{s-1}+4 \gamma \frac{s}{s-1}}{2(1-\alpha)} \geq 1,  
\]
form which one conclude that 
\[
 4 \gamma \geq 2(1-\alpha) \tfrac{s-1}{s} - \tfrac{s-2}{s} 
  = 2(1-\alpha) - \tfrac{2(1-\alpha)}{s} - 1 + \tfrac{2}{s} 
  \geq 2(1-\alpha) - 1 = 1 - 2 \alpha 
\]
and thus taking into consideration that $\beta \leq \frac{1}{2}$ 
\[
   2\gamma \geq  \tfrac{1}{2} - \alpha \geq  \beta - \alpha, 
\]
which was to proved. 

(ii) Second suppose that \underline{$\vartheta < 1$}. Then 
$\vartheta = \frac{\frac{s-2}{s-1}+4 \gamma \frac{s}{s-1}}{2(1-\alpha)}$ 
and thus 
\begin{align*}
 2\gamma - (\beta - \alpha)\vartheta &= 2\gamma - (\beta - \alpha) \tfrac{\frac{s-2}{s-1}+4 \gamma \frac{s}{s-1}}{2(1-\alpha)}
\\
 &= 2\gamma \,\big(1 - \tfrac{\beta - \alpha}{1-\alpha}\tfrac{s}{s-1}\big) - \tfrac{\beta - \alpha}{2(1-\alpha)}\tfrac{s-2}{s-1}  
 \geq 0, 
\end{align*}
where the last inequality follows by~\eqref{eq: gamma}.  

This concludes the proof of Theorem~\ref{thm for nonparam adjustment}.  

\smallskip

\subsection{Proof of Corollary~\ref{cor for nonparam adjustment}} The assertion follows 
simply  from the fact that 
if $4\gamma + 2 \alpha \frac{s}{s-1} \geq 1$, then $\vartheta = 1$ and 
$\Jtilde_{n}(\epsilon) = J_{n}(\epsilon)$ (see \eqref{eq: Jn epsilon}).

\smallskip 

\bibliographystyle{apalike}
\bibliography{short,ReferencesU}

\end{document}